\documentclass[11pt,leqno]{amsart}
\usepackage{amsmath,amssymb,amsthm}
\usepackage{hyperref}
\usepackage{xcolor}
\usepackage{soul}

\usepackage{tikz}
\usepackage{bbm}

\newcommand{\R}{\mathbb{R}}
\newcommand{\N}{\mathbb{N}}

\DeclareMathOperator{\dist}{dist\,}
\DeclareMathOperator{\diam}{diam\,}
\DeclareMathOperator{\co}{co}

\newcommand{\cconv}{\overline{\co}}

\renewcommand{\geq}{\geqslant}
\renewcommand{\leq}{\leqslant}

\newcommand{\norm}[1]{\left\Vert#1\right\Vert}

\newcommand{\Lip}{{\mathrm{Lip}}_0}

\newcommand{\sign}{\operatorname{sign}}

\newcommand{\F}[1]{\mathcal{F}(#1)}

\newcommand{\cco}{\overline{\operatorname{co}}}
\newcommand{\ext}[1]{\operatorname{ext}\left(#1\right)}

\newcommand{\preext}[1]{\operatorname{pre-ext}\left(#1\right)}
\newcommand{\dent}[1]{\operatorname{dent}\left(#1\right)}
\newcommand{\Mol}[1]{\operatorname{Mol}\left(#1\right)}

\newcommand{\spann}{\mathop{\mathrm{span}}\nolimits}

\newcommand{\abs}[1]{\left|{#1}\right|}
\newcommand{\set}[1]{\left\{{#1}\right\}}
\newcommand{\pare}[1]{\left({#1}\right)}

\newcommand{\wsconv}{\stackrel{w^*}{\longrightarrow}}
\newcommand{\cl}[1]{\overline{#1}}
\newcommand{\wcl}[1]{\overline{#1}^w}
\newcommand{\wscl}[1]{\overline{#1}^{w^*}}
\newcommand{\geqn}{\succcurlyeq}

\newcommand{\ape}[1]{\operatorname{ape}\left(#1\right)}

\newtheorem{theorem}{Theorem}[section]
\newtheorem{lemma}[theorem]{Lemma}

\newtheorem{proposition}[theorem]{Proposition}
\newtheorem{corollary}[theorem]{Corollary}
\theoremstyle{definition}
\newtheorem{definition}[theorem]{Definition}
\newtheorem{example}[theorem]{Example}
\newtheorem{question}[theorem]{Question}

\theoremstyle{remark}
\newtheorem{remark}[theorem]{Remark}
\numberwithin{equation}{section}
\def\fnote#1{\footnote}

\def\ignora#1{}
\def\n3#1{\left\vert  \! \left\vert \! \left\vert \, #1 \, \right\vert \!
  \right\vert \! \right\vert }

\newcommand{\iten}{\ensuremath{\widehat{\otimes}_\varepsilon}}
\newcommand{\pten}{\ensuremath{\widehat{\otimes}_\pi}}

\begin{document}

\title{ Almost preserved extreme points }

\author[R. J. Aliaga]{Ram\'on J. Aliaga}
\address[R. J. Aliaga]{Instituto Universitario de Matem\'atica Pura y Aplicada,
Universitat Polit\`ecnica de Val\`encia,
Camino de Vera S/N,
46022 Valencia, Spain}
\email{raalva@upv.es}
\urladdr{\url{https://raalva.wordpress.com}}

\author[L.C. García-Lirola]{Luis C. Garc\'ia-Lirola}
\address[L.C. García-Lirola]{Departamento de Matemáticas, Universidad de Zaragoza, 50009, Zaragoza, Spain} 
\email{\texttt{luiscarlos@unizar.es}}
\urladdr{\url{https://personal.unizar.es/luiscarlos/}}

\author[J. Guerrero-Viu]{Juan Guerrero-Viu}
\address[J. Guerrero-Viu]{Departamento de Matemáticas, Universidad de Zaragoza, 50009, Zaragoza, Spain} 
\email{j.guerrero@unizar.es}

\author[M. Raja]{Mat\'ias Raja}
\address[M. Raja]{Departamento de Matemáticas, Universidad de Murcia, 30100, Murcia, Spain}
\email{matias@um.es}

\author[A. Rueda Zoca]{Abraham Rueda Zoca }
\address[A. Rueda Zoca]{Universidad de Granada, Facultad de Ciencias.
Departamento de An\'{a}lisis Matem\'{a}tico, 18071-Granada
(Spain)} \email{ abrahamrueda@ugr.es}
\urladdr{\url{https://arzenglish.wordpress.com}}

% 46B20 Geometry and structure of normed linear spaces
% 46B22 Radon-Nikodym, Krein-Milman and related properties
% 46B28 Spaces of operators; tensor products; approximation properties


%% Abstract
\begin{abstract}
%% Text of abstract
In this paper we introduce the notion of an almost preserved extreme point (APEP) of a set as a weakening of the concept of preserved extreme points, and we systematically study such points. As a main result, we prove that a Banach space $X$ has the Radon-Nikod\'{y}m property (RNP) if and only if every closed, convex, and bounded subset of the space has an APEP. Similarly, we prove that $X$ has the RNP if and only if the unit ball of every equivalent renorming has an APEP. We further investigate APEPs of the unit ball of classical Banach spaces, absolute sums, Lipschitz-free spaces, and projective tensor products. In the latter setting, our work also describes the preserved extreme points in the unit ball under the assumption that every bounded operator is compact, thereby partially solving an open problem.
\end{abstract}

\maketitle

\markboth{ALIAGA, GARC\'IA-LIROLA, GUERRERO-VIU, RAJA AND RUEDA ZOCA}{ALMOST PRESERVED EXTREME POINTS}

\tableofcontents

\section{Introduction}

A longstanding open question in the geometry of Banach spaces is whether the Radon-Nikod\'{y}m property (RNP) and the Krein-Milman property (KMP) are equivalent. Recall that a Banach space $X$ has the KMP if every closed, bounded, and convex subset of $X$ has an extreme point. On the other hand, one of the equivalent reformulations of the RNP is that every closed, convex and bounded subset of $X$ has a denting point (see e.g. \cite[Section VII.6]{disuhl}). Since denting points are always extreme, the implication RNP $\Rightarrow$ KMP is clear. Whether the converse implication holds has motivated a vast literature since the 1980s (see e.g. \cite{bourgain79,bourtal81,cas88,jam93,lome24,scha87} and references therein).

Observe that there exists a fundamental difference between the concept of denting point (which is a metric notion) and the concept of extreme point (which is a linear notion). This distinction explains the difficulty behind the open question whether the KMP implies the RNP. Halfway between extreme points and denting points, we have an intermediate notion which reveals a rather better interplay with the RNP.

Given a Banach space $X$ and a bounded, closed and convex subset $C$ of $X$, we say that $x_0\in C$ is a \textit{preserved extreme point} of $C$ (sometimes called \textit{weak$^*$-extreme point}) if $x_0$ is an extreme point of $\overline{C}^{w^*}$, where the weak$^*$ closure is taken in $X^{**}$. 

Clearly, if $x_0\in C$ is a preserved extreme point of $C$ then it is an extreme point. Moreover, denting points are preserved extreme points. This can be easily seen from the fact that a point $x_0\in C$ is a preserved extreme point if, and only if, given any weak neighbourhood $W$ of $C$ such that $x_0\in W$ there exists a slice $S$ of $C$ with $x_0\in S\subseteq W$ (see e.g. \cite[Chapter 0]{lctesis}).

Coming back to the RNP, since denting points are preserved extreme points, it follows that if a Banach space $X$ has the RNP then every bounded, closed, and convex subset $C$ of $X$ has a preserved extreme point. This time the converse is known to be true. For instance, \cite[Theorem 1.1]{scsewe89} establishes that if $X$ is a Banach space failing the RNP then, given $\varepsilon>0$, there exists a closed, convex and bounded subset $C$ of $X$ such that
$$\dist\left(\ext{\overline{C}^{w^*}},X\right)>\frac{1}{2}-\varepsilon,$$
where $\ext{\overline{C}^{w^*}}$ stands for the set of all the extreme points of $\overline{C}^{w^*}$. 

In order to point out the difference between the notions of denting and preserved extreme points, we remark that given a closed, bounded and convex subset $C$ of $X$ and $x_0\in C$, then:
\begin{itemize}
    \item $x_0$ is a denting point in $C$ when the slices of $C$ containing $x_0$ form a neighbourhood system of $x_0$ for the norm topology on $C$.
    \item $x_0$ is a preserved extreme point in $C$ when the slices of $C$ containing $x_0$ form a neighbourhood system of $x_0$ for the weak topology on $C$.
\end{itemize}

To obtain a more geometric description of the RNP, we will consider the following weakening of the concept of preserved extreme point. All along this text we will assume (with no explicit mention) that all the slices we consider are non-empty.

\begin{definition}
Let $X$ be a Banach space and let $C$ be a bounded, closed and convex subset of $X$. We say that a point $x_0\in C$ is an \textit{almost preserved-extreme point (APEP)} if, for every weakly open subset $W$ of $C$ containing $x_0$ there exists a slice $S$ of $C$ such that $S\subseteq W$.

The set of APEPs of $C$ will be denoted $\ape{C}$.
\end{definition}

Note that, unlike preserved extreme points, APEPs only require that every neighbourhood contains a slice, but the slice does not need to contain $x_0$. Observe also that, in our definition, we consider a weakening of the notion of preserved extreme point instead of an intermediate property between preserved and almost preserved-extreme point (thus, APEP should be read as A(PEP) instead of (AP)EP; we add a hyphen between \textit{preserved} and \textit{extreme} in order to highlight this). This motivation will be confirmed when we describe examples of APEP which fails to be extreme (c.f. e.g. Example~\ref{exam:findinapepnoextre}).

At this point it is clear that if $X$ has the RNP then every closed, convex and bounded subset of $X$ has an APEP. A natural question is whether the converse is also true. We will give an affirmative answer in Section~\ref{section:charnp}.

This paper presents an intensive study of the notion of APEPs. We now outline the contents of the paper. In Section~\ref{section:notation} we  introduce all the necessary notation and preliminary results that we need for the main sections. In Section~\ref{section:firstresults} we obtain the first results about APEPs. In Theorem~\ref{theo:caraAPEP} we prove that, for a closed, convex and bounded set $C$ of a Banach space $X$, $x_0\in C$ is an APEP if and only if $x_0$ is in the weak$^*$ closure (in $X^{**}$) of extreme points of $\overline{C}^{w^*}$. This characterisation will be exploited throughout the text. In Section~\ref{section:classicalspaces} we study APEPs in some classical Banach spaces. We characterise the set of APEPs of the unit ball of $L_p(\mu)$ spaces for $1<p<\infty$ (Example~\ref{exam:lpinfidim}), in $L_1(\mu)$ (Theorem~\ref{theo:charapepL1}) and in $C(K)$ spaces (Theorem~\ref{theo:C(K)charapep}). We also study APEP points in the unit ball of $\ell_p$-sums of Banach spaces, for $1\leq p\leq \infty$, in Subsection~\ref{subsection:absolutesums}. 

In Section~\ref{section:charnp}, we prove that the notion of APEP has the desired interplay with the RNP and we provide new characterisations of that property. First of all, we prove in Theorem~\ref{theo:cararnpapepsubconjuntos} that a Banach space $X$ has the RNP if and only if every closed, convex, and bounded subset of $X$ has an APEP. Using this result and renorming techniques from \cite{scsewe89}, we show that if a Banach space $X$ fails the RNP then there exists an equivalent renorming of $X$ whose unit ball fails to have any APEP (Theorem~\ref{theo:renormasinapepnornp}). As a consequence we get a second characterisation of the RNP using APEP: a Banach space $X$ has the RNP if and only if the unit ball of every equivalent renorming of $X$ contains an APEP (Corollary~\ref{cor:cararnpequireno}). 

In Section~\ref{section:Lipschitzfree}, we move to the study of APEPs of the unit ball of Lipschitz-free spaces. We prove that any APEP must be either an elementary molecule or $0$. We characterise APEP molecules as the norm limit of denting points, and provide examples of APEPs that are unpreserved extreme points.

In Section~\ref{section:tensor} we make an intensive study of the APEP in projective tensor products, and more precisely in sets of the form $\cconv(C\otimes D)$ where $C\subseteq X$ and $D\subseteq Y$ are symmetric, bounded, closed, and convex subsets. Theorem~\ref{theo:tensornece} proves  that if $z$ is an APEP of $\cconv(C\otimes D)$ then $z\in C\otimes D$. Conversely, we prove that $x_0\otimes y_0$ is an APEP of $\cconv(C\otimes D)$ if either $x_0$ is denting in $C$ and $y_0$ is APEP in $D$ (Theorem~\ref{theo:tensordienteAPEP}) or if $x_0$ and $y_0$ are APEP in $C$ and $D$ respectively and $x_0\otimes y_0$ has a compact neighbourhood system for the weak topology (Theorem~\ref{theo:tensorAPEPcompactneigh}). We close the section by analysing the APEP and the preserved extreme points under the additional assumption that every operator $T\colon X\longrightarrow Y^*$ is compact. Under this assumption, we prove that if $C$ and $D$ have non-empty interior and $z\neq 0$ is a point in $\cconv(C\otimes D)$, then $z$ is an APEP if and only if $z=x_0\otimes y_0$, where $x_0$ and $y_0$ are APEPs in $C$ and $D$ respectively. Similarly, under the same assumptions, we prove that $z$ is a preserved extreme point if and only if $z=x_0\otimes y_0$ for $x_0$ and $y_0$ being preserved extreme points in $C$ and $D$ respectively. This provides a positive solution to \cite[Question 3.9]{ggmr23}, where it is asked whether $x_0\otimes y_0$ is a preserved extreme point of $B_{X\pten Y}$ if both $x_0$ and $y_0$ are preserved extreme points in $B_X$ and $B_Y$, under the assumption that every bounded operator from $X$ to $Y^*$ is compact. 

Finally, we collect in Section~\ref{section:remarks} some remarks and open questions from our work. We also present after Lemma~\ref{lemma:disw*closure} a second proof of Theorem~\ref{theo:cararnpapepsubconjuntos} with a different approach to the one exhibited in Section~\ref{section:charnp}.

\medskip
\section{Notation and preliminary results}\label{section:notation}
We will only consider real Banach spaces. Given a Banach space $X$, we denote by $B_X$ and $S_X$ its closed unit ball and unit sphere respectively. We also denote by $X^*$ the topological dual of $X$. Given $E\subseteq X$, we write $\spann(E)$ for the linear span of $E$. We denote by $L(X,Y)$, $K(X, Y)$, and  $F(X, Y)$ the spaces of bounded, compact, and finite-rank operators from $X$ to $Y$, respectively.

Given a subset $C$ of a Banach space $X$ we denote by $\co(C)$ (resp. $\cconv(C)$) the convex hull (respectively the closed convex hull) of $C$. Given $x^*\in X^*$ and $\alpha>0$, we denote by
\[ S(C, x^*, \alpha)=\{x\in C: x^*(x)>\sup x^*(C)-\alpha\}\]
the (open) slice of $C$ produced by $x^*$. If $X=Y^*$ is a dual Banach space  and $x^*\in Y\subseteq Y^{**}=X^*$, the above set will be called a \textit{weak$^*$ slice}.

\subsection{Extremal structure and Radon-Nikod\'{y}m property}

Given a subset $C$ of $X$, a point $x_0\in C$ is said to be \textit{an extreme point in $C$} if it is not the center of any non-degenerate line segment in $C$; in other words, if $x_0=\frac{y+z}{2}$ for $y,z\in C$ implies $y=z=x_0$. We denote by $\ext{C}$ the set of all the extreme points in $C$.

Let us point out here several classical results in Banach space theory which shows the importance of extreme points in compact convex sets and that will be used throughout the text without explicit reference. 

\begin{enumerate}
    \item (Krein-Milman theorem \cite[Theorem 3.65]{checos}) If $C$ is a weakly compact and convex subset of $X$ then $C=\cconv(\ext{C})$. Similarly, if $C$ is a weak* compact convex subset of $X^*$ then $C=\cconv^{w^*}(\ext{C})$.
    \item (Choquet lemma \cite[Lemma 3.69]{checos}) If $C$ is a weakly compact and convex subset of $X$, then for every $x\in \ext{C}$ the slices of $C$ containing $x$ form a neighbourhood base of $x$ in the weak topology of $C$. The same result holds for the weak$^*$ topology of $X^*$ replacing slices with weak$^*$ slices.
    \item (Milman theorem \cite[Theorem 3.66]{checos}) If $C$ is a weakly compact and convex subset of $X$ and $B\subseteq C$ is such that $\cconv(B)=C$, then $\ext{C}\subseteq \overline{B}^{w}$. Similarly, if $C\subseteq X^*$ is weak* compact and convex and $B\subseteq C$ is such that $\cconv^{w^*}(B)=C$, then $\ext{C}\subseteq \overline{B}^{w^*}$.
\end{enumerate}

In the sequel we will present some strengthenings of the concept of extreme point. Our main reference will be \cite[Chapter 0]{lctesis}.

Given a bounded set $C\subseteq X$, we say that $x_0\in C$ is a \textit{preserved extreme point} in $C$ if $x_0$ is an extreme point of the weak$^*$ closure of $C$ in $X^{**}$ (that is, $x_0\in \ext{\overline{C}^{w^*}})$. We denote by $\preext{C}$ the set of all preserved extreme points of $C$.

For closed, bounded and convex $C\subseteq X$, it follows from Choquet's lemma that $x_0\in\preext{C}$ if and only if slices of $C$ containing $x_0$ form a neighbourhood base of $x_0$ for the weak topology of $C$.

If, in the above characterisation, we replace the weak topology with the norm topology, we arrive at the notion of \textit{denting point}. That is, a point $x_0\in C$ is said to be a \textit{denting point} if, for every $\varepsilon>0$, there exists a slice $S$ of $C$ with $x_0\in S$ and $\diam(S)<\varepsilon$, where $\diam(S)$ stands for the diameter of $S$. We denote by $\dent{C}$ the set of all denting points of $C$.

The set of denting points plays a very important role in the Banach spaces with the \textit{Radon-Nikod\'{y}m property}. 
Let us recall that the Radon-Nikod\'{y}m property, from now on RNP, was originally defined for Banach spaces by the validity of a vector version of the classic Radon-Nikod\'{y}m theorem on derivation of measures. Namely, $X$ has the RNP if for any $\sigma$-finite measure space $(\Omega, \Sigma, \mu)$ and any $\mu$-continuous vector measure $\nu: \Sigma \rightarrow X$ 
of bounded variation, there exists a Bochner integrable function $f\colon \Omega \rightarrow X$ such that 
\begin{equation}\label{eq:deriv} \nu(A) =\int_A f \, d\mu \end{equation}
for every $A \in \Sigma$, see \cite{disuhl} or \cite{Bourgin} for details. Moreover, the RNP can be localized on closed convex subsets of $X$. We say that $C \subseteq X$ has the RNP if for any vector measure $\nu: \Sigma \rightarrow X$ that is absolutely continuous with respect to $\mu$, as before, and having average range in $C$, meaning that $\nu(A)/\mu(A) \in C$ for any $A \in \Sigma$ with $\mu(A)>0$, there is a Bochner integrable function $f\colon \Omega\to X$ satisfying \eqref{eq:deriv}. Note that for $C \subseteq X$ bounded and $\mu$ finite, the average range condition
implies 
$\mu$-continuity and bounded variation for $\nu$.
Another important observation is that the RNP can be witnessed just by the Lebesgue measure on $\R$ or $[0,1]$.

There are many characterizations of the RNP: differentiation of vector valued functions, integral representation of operators from $L^1(\mu)$ to a Banach space, convergence of vector valued martingales or descriptive topology. However, we are more interested in the geometrical characterizations of the RNP. A closed convex bounded subset $C \subseteq X$ has the RNP if and only if every non-empty (convex) subset has arbitrarily small slices, in such a case we say that $C$ is (hereditarily) \textit{dentable}. It turns out that if $C$ has the RNP, then it or any of its non-empty subsets have denting points and, even, \textit{strongly exposed points} (a denting point where the arbitrarily small slices can be taken parallel), see \cite{Bourgin} for more information.

The equivalence among the different characterizations of the RNP is far from trivial. Despite the geometrical flavor of the main notion and techniques of this paper, we will occasionally appeal to measure theoretical characterizations of the RNP to prove our results.

\subsection{Lipschitz-free spaces}\label{subsection:Lipschitz-free}

Let $(M,d)$ be a complete metric space where a distinguished ``base point'' $0\in M$ has been selected. The Lipschitz space $\Lip(M)$ is defined as the Banach space of all Lipschitz functions $f:M\to\R$ such that $f(0)=0$, endowed with the norm given by the best Lipschitz constant
$$
\norm{f} = \sup\set{\frac{f(x)-f(y)}{d(x,y)} : x\neq y\in M} .
$$
For each $x\in M$, the evaluation functional $\delta(x):f\mapsto f(x)$ belongs to the dual $\Lip(M)^*$. The \textit{Lipschitz-free space over $M$} is defined as the closed space $\F{M}:=\cl{\spann}\set{\delta(x):x\in M}$. It is not too hard to see that $\F{M}$ is, in fact, an isometric predual of $\Lip(M)$. We refer the reader to \cite{Weaver} for basic facts about Lipschitz and Lipschitz-free spaces.

The most important elements of $\F{M}$ are the so-called \textit{(elementary) molecules}, of the form
$$
m_{xy} := \frac{\delta(x)-\delta(y)}{d(x,y)}
$$
for $x\neq y\in M$. The set of molecules in $\F{M}$ will be denoted $\Mol{M}$. Molecules have norm $1$, and it follows easily from the Hahn-Banach separation theorem that $B_{\F{M}}=\cconv(\Mol{M})$ (see e.g. \cite[Proposition 3.29]{Weaver}). The weak closure of $\Mol{M}$ is either $\Mol{M}$ or $\Mol{M}\cup\set{0}$, depending on whether $M$ bi-Lipschitz embeds into some finite-dimensional Banach space or not (see \cite[Proposition 2.9]{gppr18} and \cite[Lemma 4.2]{GRZ}). Moreover, norm- and weak convergence agree on $\Mol{M}$: a net of molecules $(m_{x_sy_s})$ converges to $m_{xy}$ precisely when $x_s\to x$ and $y_s\to y$ (see e.g. \cite[Lemma 2.2]{gppr18} and \cite[Lemma 1.2]{Veeorg}).

Lipschitz-free spaces provide a convenient toolbox for the construction of Banach spaces with a predetermined extremal structure, because the various types of extremal objects in $B_{\F{M}}$ admit simple metric characterisations when $M$ is complete:
\begin{itemize}
\item All extreme points of $B_{\F{M}}$ are molecules \cite[Theorem 3.1]{APS3}.
\item A molecule $m_{xy}$ is extreme if and only if $d(x,p)+d(p,y)>d(x,y)$ whenever $p\neq x,y$ \cite[Theorem 1.1]{AliagaPernecka}.
\item Preserved extreme points and denting points of $B_{\F{M}}$ agree \cite[Theorem 2.4]{gppr18}.
\item A molecule $m_{xy}$ is denting if and only if for every $\varepsilon>0$ there is $\delta>0$ such that $d(x,p)+d(p,y)>d(x,y)+\delta$ whenever $d(x,p)>\varepsilon$ and $d(y,p)>\varepsilon$ \cite[Theorem 4.1]{AliagaGuirao}.
\end{itemize}

We will also need the following notion introduced in \cite{JRZ}: we say that a function $f\in S_{\Lip(M)}$ is \textit{local} if, for every $\varepsilon>0$, there exist $u\neq v\in M$ such that $f(m_{uv})>1-\varepsilon$ and $d(u,v)<\varepsilon$. This definition is a pointwise version of the notion of local metric space introduced in \cite{ikw07}, and it was introduced in order to study Daugavet points in $\F{M}$.

Let us point out our interests in local Lipschitz functions in the following remark for easy reference.

\begin{remark}\label{remark:slicelocaldiam}
If $f$ is local, then any slice $S(B_{\F{M}},f,\varepsilon)$ has diameter $2$. Even though this result is well known (and implicitly observed in (2)$\Rightarrow$(1) in \cite[Theorem 3.2]{JRZ}), let us briefly outline the proof for completeness: in the above situation, any such slice will contain molecules $m_{uv}$ for which $d(u,v)$ is arbitrarily small, so our claim follows from \cite[Theorem 2.6]{JRZ}.
\end{remark}

\subsection{Projective tensor products}\label{subsection:tensor}

The projective tensor product of $X$ and $Y$, denoted by $X \pten Y$, is the completion of the algebraic tensor product $X \otimes Y$ endowed with the norm
$$
\|z\|_{\pi} := \inf \left\{ \sum_{n=1}^k \|x_n\| \|y_n\|: z = \sum_{n=1}^k x_n \otimes y_n \right\},$$
where the infimum is taken over all such representations of $z$. The reason for taking the completion is that $X\otimes Y$ endowed with the projective norm is complete if and only if either $X$ or $Y$ is finite dimensional (see \cite[p. 43, Exercises 2.4 and 2.5]{Ryan}).

It is well known that $\|x \otimes y\|_{\pi} = \|x\| \|y\|$ for every $x \in X$, $y \in Y$, and that the closed unit ball of $X \pten Y$ is the closed convex hull of the set $B_X \otimes B_Y = \{ x \otimes y: x \in B_X, y \in B_Y \}$. Throughout the paper, we will use both facts without any explicit reference.

Observe that every $G\in L(X, Y^*)$ acts on $X \pten Y$ via
$$
G \left( \sum_{n=1}^{k} x_n \otimes y_n \right) = \sum_{n=1}^{k} G(x_n)(y_n),$$
for $\sum_{n=1}^{k} x_n \otimes y_n \in X \otimes Y$. This action establishes a linear isometry from $L(X,Y^*)$ onto $(X\pten Y)^*$ (see e.g. \cite[Theorem 2.9]{Ryan}). Throughout this paper we will use the isometric identification $(X\pten Y)^*= L(X,Y^*)$ without any explicit mention.

Recall that a Banach space $X$ has the \textit{approximation property} (AP)
if there exists a net $(S_\alpha)_\alpha$ in $F(X,X)$ such that $S_\alpha (x) \to x$ for all $x \in X$. It is not difficult to show that if either $X$ or $Y$ has the AP then $K(X,Y^*)$ is separating for $X\pten Y$ (c.f. e.g. \cite[Lemma 2.2]{ggmr23}).

Recall that given two Banach spaces $X$ and $Y$, the
\textit{injective tensor product} of $X$ and $Y$, denoted by
$X \iten Y$, is the completion of $X\otimes Y$ under the norm given by
\begin{equation*}
   \Vert u\Vert_{\varepsilon}:=\sup
   \left\{
      \left|\sum_{i=1}^n  x^*(x_i)y^*(y_i)
      \right| : x^*\in S_{X^*}, y^*\in S_{Y^*}
   \right\},
\end{equation*}
where $u=\sum_{i=1}^n x_i\otimes y_i$ (see \cite[Chapter 3]{Ryan} for background).
Every $u \in X \iten Y$ can be viewed as an operator $T_u \colon X^* \rightarrow Y$ which is weak$^*$-to-weak continuous. Under this point of view, the norm on the injective tensor product is nothing but the operator norm.

It is known that, given two Banach spaces $X$ and $Y$, we have $(X\iten Y)^*=X^*\pten Y^*$ if either $X^*$ or $Y^*$ has the RNP and either $X^*$ or $Y^*$ has the AP \cite[Theorem 5.33]{Ryan}.

\medskip
\section{First results and examples}\label{section:firstresults}

In this section we will provide the first general results, examples and characterisations of APEPs in Banach spaces. We begin with a simple fact that will be used in subsequent sections.

\begin{lemma}\label{lemma:apep_w_closed}
Let $X$ be a Banach space and $C\subseteq X$ be bounded, closed and convex. Then $\ape{C}$ is weakly closed.
\end{lemma}

\begin{proof}
Suppose that $x\in C$ belongs to the weak closure of $\ape{C}$; we have to prove that $x\in\ape{C}$. Let $W$ be a weak neighbourhood of $x$, and let us see that it contains a slice of $C$. By assumption, $W$ contains some $y\in\ape{C}$. Thus $W$ is a weak neighborhood of $y$ and, since $y$ is an APEP, is must contain a slice of $C$ as desired.
\end{proof}

This already provides us with examples of situations where APEPs are plentiful. One simple case is the unit ball of an infinite dimensional Banach space $X$ for which $\preext{B_X}=S_X$. Since $S_X$ is weakly dense in $B_X$ for every infinite dimensional Banach space $X$, we have the following consequence of Lemma \ref{lemma:apep_w_closed}.

\begin{proposition}\label{prop:everypointispreserved}
Let $X$ be an infinite dimensional Banach space such that $\preext{B_X}=S_X$. Then $\ape{B_X}=B_X$.
\end{proposition}

\begin{example}\label{exam:lpinfidim}
The above result applies to get that if $X$ is an infinite dimensional $L_p(\mu)$ space, for $1<p<\infty$, then every point of $B_X$ is APEP.

Moreover this shows that, unlike extreme points, APEPs of $B_X$ do not have to belong to $S_X$. In particular, in $B_X$ there are APEPs which are not extreme.
\end{example}

This applies in particular to LUR norms. Recall that the norm of a Banach space $X$ is \emph{locally uniformly rotund} (LUR) if for all $x,x_n\in X$ satisfying $\lim_n (2\Vert x\Vert^2+2\Vert x_n\Vert^2-\Vert x+x_n\Vert^2)=0$ one has $\lim_n \Vert x_n-x\Vert=0$ (see \cite[Definition 7.9]{checos}). It is not difficult to prove that if the norm of a Banach space $X$ is LUR then every point of $S_X$ is strongly exposed (in particular it is a preserved extreme point). Consequently we get the following result.

\begin{corollary}\label{coro:renorseparable}
If $X$ is an infinite dimensional Banach space whose norm is LUR, then every point of $B_X$ is APEP.

In particular, every infinite dimensional separable Banach space and every infinite dimensional reflexive space admits an equivalent renorming such that every point of the new unit ball is APEP.
\end{corollary}

Compare the last statement to Corollary \ref{cor:cararnpequireno}.

Next, we give another simple lemma that can be understood as a generalisation of Milman's theorem.

\begin{lemma}\label{lemma:milman_apep}
Let $X$ be a Banach space and $C\subseteq X$ be bounded, closed and convex. If $B\subseteq C$ is such that $\cconv(B)=C$, then $\ape{C}\subseteq\wcl{B}$.
\end{lemma}

\begin{proof}
Note that every slice $S$ of $C$ contains a point of $B$. Indeed, let $S=\set{x\in C : f(x)>\alpha}$ for some $f\in X^*$, $\alpha\in\R$. If $S\cap B=\varnothing$ then $f(x)\leq\alpha$ for all $x\in B$ and thus for all $x\in\cconv(B)=C$, so $S$ must be empty.

Now let $x$ be an APEP of $C$. Then every weak neighbourhood $W$ of $x$ contains a slice of $C$ and therefore intersects $B$, so we conclude $x\in\wcl{B}$.
\end{proof}

For some of our results, we will need the following, more precise conclusion in the context of Lemma \ref{lemma:milman_apep}.

\begin{lemma}\label{lemma:apepsmallslices} Let $X$ be a Banach space, $B\subseteq X$ a bounded set, and $C=\cconv(B)$. Let $x\in\ape{C}$. Then either
\begin{itemize}
\item[1)] there is $\varepsilon>0$ and a net $(y_\alpha)\subseteq B$ such that $y_\alpha\stackrel{w}{\to} x$ and $\norm{y_\alpha-x}\geq \varepsilon$ (i.e. $x\in \wcl{B\setminus B(x,\varepsilon)}$), or
\item[2)] for every weakly open set $W$ containing $x$ and $\varepsilon>0$ there exists a slice $S$ of $C$ such that $S\subseteq W$ and $\diam(S)<\varepsilon$.
\end{itemize}
\end{lemma}

\begin{proof}
Assume that 2) does not hold. Then we can find a weakly open set $W_0$  containing $x$ and $\varepsilon_0>0$ so that every slice $S$ of $C$ contained in $W_0$ satisfies $\diam(S)\geq \varepsilon_0$.

Now, let $W$ be a weakly open set containing $x$. Since $x$ is an APEP, we can find a slice $S=S(C, f,\alpha)$ with $S\subseteq W\cap W_0$. Take $0<\delta<1$. By \cite[Lemma~5.2.14]{lctesis} we have 
$$\varepsilon_0\leq \diam(S(C,f,\delta\alpha))\leq 2 \diam(S(C,f,\alpha)\cap B)+4\delta .$$
If we choose $\delta$ small enough to guarantee that $\frac{\varepsilon_0-4\delta}{2}>\frac{\varepsilon_0}{4}$, then it follows that there are $y,z\in S(C,f,\alpha)\cap B$ with $\norm{y-z}\geq \frac{\varepsilon_0}{4}$. Then either $\norm{y-x}\geq \frac{\varepsilon_0}{8}$ or $\norm{z-x}\geq \frac{\varepsilon_0}{8}$. 

In any case, we have proved that there is $y_W\in W\cap B$ with $\norm{y_W-x}\geq \frac{\varepsilon_0}{8}$. It is clear that $(y_W)$ is a net in $B$ that converges weakly to $x$. 
\end{proof}

\begin{remark}\label{remark:puntualizademoluisca}
Observe that \cite[Lemma~5.2.14]{lctesis} which is established for unit balls. An inspection in the proof, however, reveals that this result can be extended to general convex sets with the same proof.
\end{remark}

The following example shows that condition 2) in Lemma \ref{lemma:apepsmallslices} can really fail to hold.

\begin{example}
In \cite{ahntt16}, a Banach space $X$ isomorphic to $C([0,1])$ is constructed such that every non-empty relatively weakly open subset of $B_X$ has diameter $2$ and such that every point of $S_X$ is a preserved extreme point of $B_X$. In particular, every point of $B_X$ is APEP (Proposition~\ref{prop:everypointispreserved}).
\end{example}

We now provide a characterisation of APEPs of a given closed, bounded and convex set in terms of the extremal structure of its weak$^*$ closure.

\begin{theorem}\label{theo:caraAPEP}
Let $X$ be a Banach space, $C\subseteq X$ be closed convex and bounded and $x\in C$. Let $D:=\overline{C}^{w^*}$ be its weak$^*$ closure in $X^{**}$. The following are equivalent: 
\begin{enumerate}
\item $x\in\ape{C}$.
\item For every weak$^*$ open subset $W$ of $D$ containing $x$ there exists a weak$^*$ slice $S$ of $D$ such that $S\subseteq W$.
\item $x\in \overline{\ext{D}}^{w^*}$.
\end{enumerate}
\end{theorem}

\begin{proof}
(1)$\Rightarrow$(2). Let $W$ be a weak$^*$ open subset of $D$ such that $x\in W$. Choose a weak$^*$ open subset $V$ of $D$ such that $x\in V\subseteq \overline{V}^{w^*}\subseteq W$. Consider $U:=V\cap C$, which is a weakly open set of $C$ with $x\in U$. Since $x$ is APEP we can find a slice $S=S(C,f,\alpha)$ such that $S\subseteq U$. We claim that
$$S(D,f,\alpha)\subseteq \overline{V}^{w^*}.$$
Indeed, given any $z^{**}\in S(D,f,\alpha)\subseteq D$ we can find a net $z_s\wsconv z^{**}$ such that $z_s\in C$ holds for every $s$. Since $f(z_s)\rightarrow f(z^{**})>\sup_D{f} -\alpha=\sup_C{f}-\alpha$ we can find an index $s_0$ such that $f(z_s)>\sup_C{f}-\alpha$ holds for every $s\geqn s_0$. Since $z_s\in C$ we infer that $z_s\in S(C,f,\alpha)\subseteq U\subseteq V$ holds for every $s\geqn s_0$. Since $z_s\wsconv z^{**}$ we infer that $z^{**}\in \overline{V}^{w^*}$, as desired. Since $\overline{V}^{w^*}\subseteq W$ we get that $S(D,f,\alpha)\subseteq W$, as required.

(2)$\Rightarrow$(3). Select a relatively weak$^*$ open subset $W$ of $D$ containing $x$ and let us find an extreme point $e^{**}$ of $D$ such that $e^{**}\in W$. By (2) we can find a weak$^*$ slice $S$ of $D$ such that $S\subseteq W$. Now the Krein-Milman theorem ensures that $D=\cconv^{w^*}(\ext{D})$ (observe that $D$ is weak$^*$ compact since $C$ is bounded). Since $S$ is a weak$^*$ slice of $D$ we get that $\emptyset\neq S\cap \ext{D}\subseteq W\cap \ext{D}$. We have proved that every weak$^*$ neighbourhood of $x$ contains an extreme point of $D$, so the implication is proved.

(3)$\Rightarrow$(1). Let $W$ be a weakly open subset of $C$ containing $x$, and let us prove that there exists a slice $S$ of $C$ with $S\subseteq W$. 
In order to do so define $\tilde W$ as the weak$^*$ open subset of $D$ defined by $W$, that is, such that $\tilde W\cap C=W$. By (3) there exists an extreme point $e^{**}$ of $D$ such that $e^{**}\in \tilde W$. Since $e^{**}$ is an extreme point of $D$, Choquet's lemma implies that there exists a slice $S=S(D,f,\alpha)$ such that $e^{**}\in S\subseteq \tilde W$.
It is now clear that $S(C,f,\alpha)\subseteq \tilde W\cap C=W$, and the proof is finished.
\end{proof}

With Theorem~\ref{theo:caraAPEP} we can now provide an example of an APEP which is not an extreme point in finite-dimensional Banach spaces.

\begin{example}\label{exam:findinapepnoextre}
Let $C\subseteq \mathbb R^3$ be a compact set whose set of extreme points is not closed (c.f. e.g. \cite[Exercise 3.144]{checos}), and set $x_0\in \overline{\ext{C}}\setminus \ext{C}=\overline{\preext{C}}\setminus\preext{C}$ (since in the finite dimensional framework clearly every extreme point is preserved). Then $x_0$ is an APEP of $C$ which is not an extreme point of $C$. 
\end{example}

An example of an extreme point which is not an APEP will be obtained in Example~\ref{exam:extremenoapepfree}. Moreover, an example of an extreme point which is an APEP but fails to be a preserved extreme point will be exhibited in Example~\ref{example:extreapepnopreser}.

\begin{remark}\label{remark:dim2apep}
It is a well-known result that if $X$ is a 2-dimensional Banach space and $C\subseteq X$ is closed, convex, and bounded then the set $\ext{C}$ is closed (c.f. e.g. \cite[Exercise 3.144]{checos}). Consequently, in such a situation every extreme point of $C$ is an APEP.
\end{remark}

Let us provide another example, which shows that if a closed, convex, and bounded set has the property that every point is APEP then this does not necessarily imply that every point is an extreme point, even if the set of extreme points is norm-dense.

\begin{example}\label{example:denseapepnoextreme}
In \cite{poulsen} a compact convex set $K\subseteq \ell_2$ is constructed with the property that the set of all extreme points is dense in $K$. In such a compact set, all points are APEP but there are (densely many) points which are not extreme points.
\end{example}

\medskip
\section{APEP in classical Banach spaces}\label{section:classicalspaces}

In this section, we pursue a characterisation of the APEPs of the unit ball of classical Banach spaces (namely $C(K)$ spaces and $L_p(\mu)$ for $1\leq p\leq \infty$). We will also study APEPs in absolute sums of Banach spaces.

It is immediate that if $X=\ell_p^n$ for $n\in\mathbb N$ and $1\leq p\leq \infty$ then the set of (preserved) extreme points of $B_X$ is closed. Consequently, the APEPs of $B_X$ coincide with the extreme points.
On the other hand, if $1<p<\infty$ and $X=L_p(\mu)$ is infinite-dimensional, Example \ref{exam:lpinfidim} shows that all points of $B_X$ are APEPs. The rest of the section will focus on the remaining cases.

\subsection{\texorpdfstring{$C(K)$}{C(K)} spaces}

Let $K$ be a compact Hausdorff topological space. In this subsection, we aim to provide a description of those points of $B_{C(K)}$ which are APEPs. The main result is the following.

\begin{theorem}\label{theo:C(K)charapep}
Let $K$ be a compact Hausdorff topological space. Let $f\in B_{C(K)}$. The following are equivalent:
\begin{enumerate}
    \item $f$ is an APEP of $B_{C(K)}$.
    \item $\vert f(t)\vert=1$ holds for every $t\in K$.
    \item $f$ is an extreme point of $B_{C(K)}$.
    \item $f$ is a preserved extreme point of $B_{C(K)}$.
\end{enumerate}
\end{theorem}

\begin{proof} (2)$\Rightarrow$(3) is straightforward, whereas (3)$\Rightarrow$(4) is well known (c.f. e.g. \cite[p. 295]{phelps}). Moreover, (4)$\Rightarrow$(1) is general. It remains to show (1)$\Rightarrow$(2). To this end, it will suffice to prove that for any $t_0\in K$ and any slice $S$ of $B_{C(K)}$ there exists some function $\varphi\in S$ with $\vert \varphi(t_0)\vert=1$. Indeed, it follows that the weakly open subset 
$$W:=\{g\in B_{C(K)}: \vert \delta_{t_0}(g)\vert<1\}$$
cannot contain any slice of $B_{C(K)}$ and therefore cannot contain any APEP either.

In order to do so, fix $t_0\in K$ and a slice $S$ of $B_{C(K)}$. We may assume that $S=S(B_{C(K)},\mu,\alpha)$ where $\mu\in C(K)^*=M(K)$ is a regular Borel measure with $\norm{\mu}=1$. Observe that we can write $\mu=\lambda \delta_{t_0}+\nu$, where $\lambda\in\R$ and $\nu\in M(K)$ is such that $\nu(\{t_0\})=0$. Now select $h\in B_{C(K)}$ such that $\mu(h)>1-\frac{\alpha}{8}$. Since $\nu$ is a regular measure and $\nu(\{t_0\})=0$, there exists an open subset $U\subseteq K$ with $t_0\in U$ such that $\vert \nu\vert(U)\leq\frac{\alpha}{8}$. Consider a Urysohn function $g\in S_{C(K)}$ such that $g(t)=0$ if $t\notin U$, $0\leq g\leq 1$ and $g(t)=1$ on $\overline V$ for some open set $V$ such that $t_0\in V\subseteq \overline{V}\subseteq U$. Now take another Urysohn function $j\in S_{C(K)}$ such that $j(t)=0$ if $t\notin V$ and $j(t_0)=1$. 

We consider $\sign(\lambda)=|\lambda|/\lambda$ if $\lambda\neq 0$ and $\sign(0)=1$. Now, define
$$\varphi:=(1-g)h+\sign(\lambda) j.$$
It is clear that $\varphi(t_0)=\sign(\lambda)\in \{-1,1\}$, so we only have to prove that $\varphi\in S$. Let us start by proving that $\Vert \varphi\Vert_{\infty}\leq 1$. Select any $t\in K$. Now we have two possibilities:
\begin{enumerate}
    \item If $t\notin V$ then $j(t)=0$, so $\vert \varphi(t)\vert=\vert 1-g(t)\vert \vert h(t)\vert\leq \vert h(t)\vert\leq 1$.
    \item If $t\in V$ then $g(t)=1$ and thus $\vert \varphi(t)\vert=\vert j(t)\vert\leq 1$.
\end{enumerate}
In any case we get $\vert \varphi(t)\vert\leq 1$. It remains to estimate $\mu(\varphi)$. Observe that
$$\mu(\varphi)=\mu((1-g)h)+\mu(\sign(\lambda)j).$$
On the one hand, since $\delta_{t_0}(h(1-g))=0$ we get that $\mu((1-g)h)=\nu((1-g)h)$. Now
\[
\begin{split}
\nu((1-g)h)=\int_K (1-g)h\ d\nu& =\int_{K\setminus U}(1-g)h d\nu+\int_U (1-g)h\ d\nu\\
& \geq \int_{K\setminus U} h\ d\nu-\Vert (1-g)h\Vert_\infty \vert \nu\vert(U)\\
& \geq \int_K h\ d\nu-\int_U h\ d\nu-\vert \nu\vert(U)\\
& \geq \nu(h)-2\vert \nu\vert(U)\geq \nu(h)-\frac{\alpha}{4}.
\end{split}
\]
On the other hand,
\[
\begin{split}
\mu(\sign(\lambda)j)& =\sign(\lambda)\lambda j(t_0)+\nu(\sign(\lambda)j)\\
& = \vert \lambda\vert+\sign(\lambda)\int_K j\ d\nu \\
& = \vert \lambda\vert+\sign(\lambda)\int_V j\ d\nu\\
& \geq \vert \lambda\vert-\vert \nu\vert(V)\geq \vert \lambda\vert-\vert \nu\vert (U)\\
& \geq \vert \lambda\vert-\frac{\alpha}{8}.
\end{split}
\]
Putting everything together we infer
$$\mu(\varphi)\geq \nu(h)-\frac{\alpha}{4}+\vert\lambda\vert-\frac{\alpha}{8}=\vert \lambda\vert+\nu(h)-\frac{3\alpha}{8}.$$
Taking into account that $\vert \lambda\vert\geq \lambda \delta_{t_0}(h)$ we clearly get that
$$\mu(\varphi)\geq (\lambda\delta_{t_0}+\nu)(h)-\frac{3\alpha}{8}=\mu(h)-\frac{3\alpha}{8}>1-\frac{\alpha}{8}-\frac{3\alpha}{8}>1-\alpha.$$
This implies that $\varphi\in S$ and the proof is finished.
\end{proof}

Let us show an immediate consequence of Theorem \ref{theo:C(K)charapep}, for describing APEP in the unit ball of  $L_\infty$ spaces.

\begin{corollary}
    Let $(\Omega, \Sigma, \mu)$ be a measure space and let $I$ be an arbitrary set. Let $x=(x_i)_{i\in I}\in B_{\ell_\infty(I)}$ and $f\in B_{L_\infty(\mu)}$. Then, 
    \begin{itemize}
        \item[a)] $f\in\ape{B_{L_\infty(\mu)}}$ if and only if $|f(\omega)|=1$ for  a.e. $\omega\in \Omega$.
        \item[b)] $x\in\ape{B_{\ell_\infty(I)}}$ if and only if $|x_i|=1$ for all $i\in I$.
    \end{itemize}
\end{corollary}

\begin{proof}
    a) We have $L_\infty (\mu)=C(K_\mu)$ isometrically, where $K_\mu$ is the maximal ideal space of $L_\infty(\mu)$ (see e.g. \cite[Theorem 9.6]{Zhu}). The result follows from Theorem \ref{theo:C(K)charapep} and the well-known characterization of extreme points of $B_{L_\infty(\mu)}$.
    
    b) is a particular case of a) taking $\mu$ as the counting measure on $I$, but can also be justified directly: we have $\ell_\infty(I)=C(\beta I)$ isometrically, where $\beta I$ is the Stone-Čech compactification of $I$, and the result follows again from Theorem \ref{theo:C(K)charapep} and the well-known characterization of extreme points of $B_{\ell_\infty(I)}$.
\end{proof}

\subsection{\texorpdfstring{$L_1(\mu)$}{L1} spaces}

In this section we aim to characterise when $f\in B_{L_1(\mu)}$ is an APEP, for a given measure space $(\Omega, \Sigma,\mu)$. We limit our analysis to localisable measure spaces, which are precisely those for which $L_1(\mu)^*$ is isometrically isomorphic to $L_{\infty}(\mu)$ \cite[Theorem 243G]{fremlin}. This is no loss of generality, as it is known that every $L_1(\mu)$ space is isometrically isomorphic to an $\ell_1$-sum of spaces of the form $L_1(\mu_i)$ where $\mu_i$ is a finite, hence localisable, measure (c.f. e.g. \cite[p. 501]{deflo}). We will deal with APEPs in $\ell_1$-sums of Banach spaces in Section \ref{subsection:absolutesums}.

Before we proceed, let us introduce a bit of notation. Recall that a measurable set $A\subseteq \Omega$ is called an \emph{atom} of $\mu$ if $\mu(A)>0$ and if $\mu(B)=0$ for every measurable subset $B\subseteq A$ such that $\mu(B)<\mu(A)$. As a consequence of \cite[Theorem~2.1]{johnson} we can decompose $L_1(\mu)$ as 
\begin{equation}\label{eq:Maharam}
    L_1(\mu)=L_1(\nu)\oplus_1 \ell_1(I),
\end{equation} 
where $\nu$ is the continuous part of $\mu$, and $I$ is the set of all atoms of $\mu$ (up to a measure 0 set).

With the above description in mind, we will first analyse the APEPs of the unit ball of $L_1(\mu)$  in the case that $\mu$ either fails to have any atom or in the case that $\mu$ is purely atomic, and then we will complete the information with the stability results of the APEPs in $\ell_1$-sums of spaces of the next section (see Proposition~\ref{prop:APEPlusumas}). In order to do so, let us start with the atomless case.

\begin{proposition}\label{prop:L1}
Let $(\Omega,\Sigma,\mu)$ be a localisable measure space such that $\mu$ is atomless. Then $\ape{B_{L_1(\mu)}}$ is empty.
\end{proposition}

\begin{proof}
In this proof we will denote $L_1=L_1(\mu)$ and $L_\infty=L_\infty(\mu)=L_1(\mu)^*$. Let $f\in B_{L_1}$ and let us prove that $f$ is not an APEP. In order to do so, let us begin with the case $f\neq 0$. Since $f=f^+-f^-$, we may assume without loss of generality that $f^+\neq 0$. Since $\int_\Omega f^+ d\mu\neq 0$ and $\mu$ is atomless we can find a subset $A\subseteq \Omega$ and $\alpha>0$ such that
$$\alpha<\int_A f^+\ d\mu=\int_A f\ d\mu<\frac{1}{2}.$$
Define $g:=\chi_{A}\in S_{L_\infty}$ and set
$$W:=\left\{ \varphi\in B_{L_1}: \alpha<g(\varphi)=\int_A \varphi\ d\mu<\frac{1}{2}\right\}.$$
Observe that $f\in W$. Indeed, 
$$g(f)=\int_\Omega f\chi_A\ d\mu=\int_A f\ d\mu=\int _A f^+\ d\mu\in \left(\alpha,\frac{1}{2}\right).$$
Let us now prove that the relatively weakly open set $W$ cannot contain any slice of $B_{L_1}$.

Indeed, take a slice $S=S(B_{L_1}, h,\beta)$ for $\beta>0$ and $h\in S_{L_\infty}$. 
By the very definition of essential supremum, there exists $\xi\in \{-1,1\}$ and $B\in \Sigma$ such that $0<\mu(B)<\infty$ and 
$$B\subseteq\{t\in \Omega: \xi h(t)>1-\beta\}.$$
Now we have two different possibilities:
\begin{itemize}
    \item[a)] If $\mu(A\cap B)\neq 0$ then define the function $\varphi:=\xi \frac{\chi_{A\cap B}}{\mu(A\cap B)}$. On the one hand we have
    \[
    \begin{split}
    h(\varphi)=\int_\Omega h \varphi\ d\mu=\frac{1}{\mu(A\cap B)}\int_{A\cap B} \xi h(t)\ d\mu(t)>1-\beta
    \end{split}
    \]
    since $\xi h(t)=\vert h(t)\vert$ on $A\cap B$. This implies $\varphi\in S$. On the other hand,
    \[
    \begin{split}
    g(\varphi)=\int_\Omega g\varphi\ d\mu=\frac{\xi}{\mu(A\cap B)}\int_{A\cap B}g\ d\mu=\xi
    \end{split}
    \]
since $g=1$ on $A\cap B$. In particular $g(\varphi)$ is either $1$ or $-1$, so $\varphi\notin W$. 
\item[b)] If $\mu(A\cap B)=0$ then define $\varphi:=\xi \frac{\chi_B}{\mu(B)}$. As before, $h(\varphi)>1-\beta$ (i.e. $\varphi\in S$), but clearly $g(\varphi)=\int_A \xi\frac{\chi_B}{\mu(B)}\ d\mu=0$ since $\mu(A\cap B)=0$. This proves that $\varphi\notin W$, as desired.
\end{itemize}

To finish the proof, it remains to be proved that $0$ is also not an APEP of $B_{L_1}$. In order to do so, define
$$W:=\left\{\varphi\in B_{L_1}: \left\vert \int_\Omega \varphi\ d\mu\right\vert<\frac{1}{2} \right\}.$$
It is immediate that $W$ is a relatively weakly open set containing $0$. However, it does not contain any slice of $B_{L_1}$. Indeed, given any slice $S$ of $B_{L_1}$, by the proof of the above case we can find $\xi\in \{-1,1\}$ and $C\subseteq \Omega$ with $\mu(C)>0$ such that $\varphi:=\xi\frac{\chi_C}{\mu(C)}\in S$. However,
\[
\begin{split}
\int_\Omega \varphi\ d\mu=\frac{\xi}{\mu(C)}\int_\Omega\chi_C\ d\mu=\xi \notin \left( -\frac{1}{2},\frac{1}{2}\right) .
\end{split}
\]
\end{proof}

Now we move to the purely atomic case, obtaining the following result.

\begin{proposition}\label{prop:l1purelyatomic}
Let $I$ be a non-empty set and consider $X=\ell_1(I)$. Then the APEPs of $B_{\ell_1(I)}$ are the standard basis vectors $\pm e_i$, $i\in I$.
\end{proposition}

\begin{proof}
Since $B_{\ell_1(I)}=\cconv(\{\pm e_i: i\in I\})$ we infer from Lemma \ref{lemma:milman_apep} that if $x$ is an APEP of $B_{\ell_1(I)}$ then $x\in \overline{\{\pm e_i: i\in I\}}^{w} = \{\pm e_i: i\in I\}$. In the opposite direction, it is clear that every element of the form $\pm e_i$ is a denting point and therefore an APEP.
\end{proof}

Now we can give a description of the APEPs in an $L_1(\mu)$ space.

\begin{theorem}\label{theo:charapepL1}
Let $(\Omega,\Sigma,\mu)$ be a localisable measure space and let $f\in B_{L_1(\mu)}$. The following are equivalent:
\begin{enumerate}
    \item $f$ is an APEP of $B_{L_1(\mu)}$.
    \item $f$ is a denting point of $B_{L_1(\mu)}$.
    \item $f=\pm \frac{\chi_A}{\mu(A)}$, where $A$ is an atom of $\mu$.
\end{enumerate}
\end{theorem}

\begin{proof}
It is well known that $(3)\Rightarrow (2)\Rightarrow(1)$. We prove $(1)\Rightarrow (3)$. According to the decomposition in \eqref{eq:Maharam} we see $f=(g,h)\in L_1(\nu)\oplus_1 \ell_1(I)$. 

Proposition~\ref{prop:APEP_finite_l1sum} below together with the preceding paragraph yield that $f$ is an APEP if, and only if, either $g=0$ and $h\in\ape{B_{\ell_1(I)}}$ or $h=0$ and $g\in\ape{B_{L_1(\nu)}}$. However the latter is impossible due to Proposition~\ref{prop:L1}. Consequently, $f$ is an APEP if, and only if, $f=(0,h)$, where $h$ is an APEP in $B_{\ell_1(I)}$. But now Proposition~\ref{prop:l1purelyatomic} implies that the above holds true if, and only if, $h=\pm e_i$ for some $i\in I$. Now, taking into account the identification of $\ell_1(I)$ with the purely atomic measures in \eqref{eq:Maharam}, the result follows.
\end{proof}

\subsection{Absolute sums of Banach spaces}\label{subsection:absolutesums}

Now, we focus on studying APEPs of the unit ball in $\ell_p$-sums of Banach spaces. We start with the case $p=1$. For finite sums, we have an easy characterisation.

\begin{proposition}\label{prop:APEP_finite_l1sum}
Let $Y,Z$ be Banach spaces. Then
$$
\ape{B_{Y\oplus_1 Z}} = (\ape{B_Y}\times\set{0}) \cup (\set{0}\times\ape{B_Z}) .
$$
\end{proposition}

\begin{proof}
Let $X=Y\oplus_1 Z$, then $X^{**}=Y^{**}\oplus_1 Z^{**}$ and therefore 
$$\ext{B_{X^{**}}}=(\ext{B_{Y^{**}}}\times\set{0}) \cup (\set{0}\times \ext{B_{Z^{**}}}).$$
Clearly, $Y^{**}\times\set{0}\subseteq X^{**}$ is weak$^*$-weak$^*$-homeomorphic to $Y^{**}$ (and similarly for $\set{0}\times Z^{**}$), so
$$
\wscl{\ext{B_{X^{**}}}}=(\wscl{\ext{B_{Y^{**}}}}\times\set{0}) \cup (\set{0}\times \wscl{\ext{B_{Z^{**}}}}) .
$$
The result now follows immediately from Theorem \ref{theo:caraAPEP}.
\end{proof}

The argument of Proposition \ref{prop:APEP_finite_l1sum} extends seamlessly (or by induction) to finite $\ell_1$-sums. For infinite sums, however, the bidual does not admit such a simple expression so a different argument is needed. In that case, we are able to characterise all non-zero APEPs.

\begin{proposition}\label{prop:APEPlusumas}
Let $X$ be the $\ell_1$-sum of a family $\set{X_i:i\in I}$ of Banach spaces.
\begin{itemize}
\item[a)] If $x=(x_i)\neq 0$ is an APEP of $B_X$, then there exists an index $j\in I$ such that $x_j\in\ape{B_{X_j}}$ and $x_i=0$ for all $i\neq j$.
\item[b)] If $x_j\in\ape{B_{X_j}}$, then the element $x=(u_i)$ defined as $u_j=x_j$ and $u_i=0$ for $i\neq j$ is an APEP of $B_X$.
\end{itemize}
\end{proposition}

\begin{proof}
a) Let $x=(x_i)$ be an APEP of $B_X$, and suppose that there are two indices $j_1\neq j_2\in I$ such that $x_{j_1},x_{j_2}$ are non-zero. For $k=1,2$ let $\varphi_k\in S_{X_{j_k}^*}$ be such that $\varphi_k(x_{j_k})=\norm{x_{j_k}}$, and consider the set
$$
W = \set{z=(z_i)\in B_X \,:\, \varphi_1(z_{j_1})>0 \text{ and } \varphi_2(z_{j_2})>0} .
$$
Then $W$ is a relatively weakly open neighbourhood of $x$. Since $x$ is an APEP, $W$ must contain a slice of the form $S=S(B_X,f,\alpha)$ for some $f\in S_{X^*}$ and $\alpha>0$. Identify $X^*$ with the $\ell_\infty$-sum of the spaces $\set{X_i^*:i\in I}$ and write $f=(f_i)$. Suppose that $\norm{f_j}>1-\alpha$ for some $j\in I$, choose $y\in B_{X_j}$ such that $f_j(y)>1-\alpha$ and let $z=(z_i)$ be such that $z_j=y$ and $z_i=0$ for $i\neq j$. Then $f(z)=f_j(y)>1-\alpha$ so $z\in S\subseteq W$, but this is not possible as either $z_{j_1}=0$ or $z_{j_2}=0$. This shows that $\norm{f_j}\leq 1-\alpha$ for all $j\in I$, contradicting $\norm{f}=1$. Hence, there must exist $j\in I$ such that $x_i=0$ for all $i\neq j$.

Let us see that $x_j$ is an APEP of $B_{X_j}$ if $x_j\neq 0$. Let $W\subseteq B_{X_j}$ be a relatively weakly open neighbourhood of $x_j$, and let $\varphi\in X_j^*$ such that $\varphi(x_j)=\norm{x_j}>0$. Define
$$W'=\left\{ z=(z_i)\in B_X : z_j\in W, \ \varphi(z_j)>0 \right\},$$
which is clearly a relatively weakly open subset of $B_X$ containing $x$. Since $x$ is an APEP, there is a slice $S=S(B_X,f,\alpha)$, for some $f=(f_i)\in S_{X^*}$ and $\alpha>0$, such that $S\subseteq W'$. We claim that $\norm{f_i}\leq 1-\alpha$ for all $i\neq j$. If not, there is some $k\neq j$ such that $\norm{f_k}>1-\alpha$, and so we can find $y\in B_{X_k}$ such that $f_k(y)>1-\alpha$. Hence, the point $z=(z_i)\in B_X$ such that $z_i=0$ for all $i\neq k$ and $z_k=y$ satisfies $f(z)=f_k(y)>1-\alpha$. Thus, $z\in S\subseteq W'$ but $\varphi(z_j)=0$, obtaining a contradiction. Therefore, $\norm{f_i}\leq 1-\alpha$, for all $i\neq j$, and since $X^*=\pare{\bigoplus X^*_i}_\infty$, it follows that $\norm{f_j}=1$. Now, define the slice $S_j=\{ y\in B_{X_j} : f_j(y)>1-\alpha\}$. If $y\in S_j$, then the point $z=(z_i)\in B_X$ such that $z_i=0$ for $i\neq j$ and $z_j=y$ satisfies $z\in S\subseteq W'$. Thus, $y=z_j\in W$, from which we conclude that $S_j\subseteq W$ and $x_j$ is therefore an APEP.

b) Let $x=(x_i)$ be a point in $B_X$ such that there is $j\in I$ with $x_i=0$ for $i\neq j$ and $x_j\in\ape{B_{X_j}}$. Let us show that $x\in\ape{B_X}$. Let $W$ be a relatively weakly open neighbourhood of $x$ in $B_X$. We may assume that $W$ is a basic weakly open set of the form $$W=\{ y\in B_X : |f^1(y-x)|<\varepsilon, \ldots, |f^n(y-x)|<\varepsilon\}$$
for some functionals $f^1=(f^1_i), \ldots, f^n=(f^n_i)$ in the unit ball of $ X^*=\pare{\bigoplus X^*_i}_\infty$ and some $\varepsilon>0$. Now, consider $$W_j =\left\{ z\in B_{X_j} : |f^1_j(z-x_j)|<\frac{\varepsilon}{2}, \ldots, |f^n_j(z-x_j)|<\frac{\varepsilon}{2} \right\},$$
which is a relatively weakly open neighbourhood of $x_j$ in $B_{X_j}$. Since $x_j$ is APEP and $x_j\in W_j$, we can find a slice $S_j=\{ z\in B_{X_j} : g(z)>1-\alpha\}$ for some $g\in S_{X_j^*}$ and $\alpha>0$, such that $S_j\subseteq W_j$. We can always assume that $\alpha\leq\frac{\varepsilon}{2}$ since $S_j$ is non-empty for all $\alpha>0$. Finally, consider the slice $S=\{ y=(y_i)\in B_X : g(y_j)>1-\alpha\}$. It is clear that if $z\in S_j\neq \emptyset$, then the point $y=(y_i)\in B_X$ such that $y_i=0$ for $i\neq j$ and $y_j=z$, satisfies that $y\in S$, so $S\neq \emptyset$. Furthermore, pick $y=(y_i)\in S$, and denote by $\hat{y}=(\hat{y}_i)$ the element of $X$ such that $\hat{y}_i=y_i$ for $i\neq j$ and $\hat{y}_j=0$. It is clear that $\norm{y_j}\geq g(y_j)>1-\alpha\geq1-\frac{\varepsilon}{2}$, so $\norm{\hat{y}}=\norm{y}-\norm{y_j}<\frac{\varepsilon}{2}$. 
Therefore, we have
$$|f^p(y-x)|\leq |f^p_j(y_j-x_j)| + |f^p(\hat{y})|<\frac{\varepsilon}{2}+\frac{\varepsilon}{2}=\varepsilon, \quad  \ 1\leq p \leq n,$$
since $y_j\in S_j\subseteq W_j$. This proves that $y\in W$, from which we conclude that $S\subseteq W$. Thus $x$ is an APEP of $B_X$.
\end{proof}

The only case that remains unclear is whether it is possible to have $0\in\ape{B_X}$ when $0\notin\ape{B_{X_i}}$ for all $i\in I$.

\medskip

Next, we consider $\ell_p$-sums of Banach spaces for $1<p<\infty$. It is easier to use arguments based on Theorem~\ref{theo:caraAPEP} in this case as, given a family $\set{X_i:i\in I}$ of Banach spaces, one has $(\bigoplus_i X_i)_p^{**}=(\bigoplus_i X_i^{**})_p$ and
\begin{equation}\label{eq:extreme_ell_p}
\ext{B_{(\bigoplus X_i)_p}} = \set{(x_i)\in S_{(\bigoplus X_i)_p} \,:\, \forall i, x_i=0 \text{ or } \frac{x_i}{\norm{x_i}}\in\ext{B_{X_i}}}.
\end{equation}
Using these descriptions, we can begin with the following necessary condition for APEPs in the unit sphere.

\begin{lemma}\label{lemma:sumapnececondi}
Let $X$ and $Y$ be two Banach spaces and let $(x_0,y_0)\in S_{X\oplus_p Y}$ be an APEP of $B_{X\oplus_p Y}$. Then either $x_0=0$ or $\frac{x_0}{\Vert x_0\Vert}$ is an APEP of $B_X$.\end{lemma}

\begin{proof}
Assume that $x_0\neq 0$ and let us prove that $\frac{x_0}{\Vert x_0\Vert}$ is an APEP. Since $(x_0,y_0)$ is an APEP of $B_{X\oplus_p Y}$ we can find by virtue of Theorem~\ref{theo:caraAPEP} a net $(e_s,f_s)$ of extreme points of $B_{(X\oplus_p Y)^{**}}=B_{X^{**}\oplus_p Y^{**}}$ such that $(e_s,f_s)\stackrel{w^*}{\to} (x_0,y_0)$. This implies that both $(e_s)\stackrel{w^*}{\to} x_0$ and $(f_s)\stackrel{w^*}{\to} y_0$. Since both $(e_s)$ and $(f_s)$ are bounded nets we can assume, up to taking subnets, that $\Vert e_s\Vert\rightarrow \lambda$ and $\Vert f_s\Vert\rightarrow \mu$. The weak* lower semicontinuity of the norm of $X^{**}$ and $Y^{**}$ implies $\Vert x_0\Vert\leq \lambda$ and $\Vert y_0\Vert\leq \mu$. Note that
\[ 1 = \norm{x_0}^p + \norm{y_0}^p \leq \lambda^p + \mu^p = \lim_s \norm{e_s}^p+\norm{f_s}^p = \lim_s \norm{(e_s,f_s)}^p \leq 1\]
so indeed $\norm{x_0}=\lambda$ and $\norm{y_0}=\mu$.

Now, up to taking a further subnet, since $\Vert e_s\Vert\rightarrow \lambda=\Vert x_0\Vert>0$ we can assume that $e_s\neq 0$ holds for every $s$. Since $(e_s,f_s)$ is an extreme point of $B_{X^{**}\oplus_p Y^{**}}$ and $e_s\neq 0$ we infer $\frac{e_s}{\Vert e_s\Vert}\in \ext{B_{X^{**}}}$ for every $s$. Since $e_s\stackrel{w^*}{\to} x_0$ and $\Vert e_s\Vert\rightarrow \Vert x_0\Vert$ we get that
$$\frac{e_s}{\Vert e_s\Vert}\wsconv \frac{x_0}{\Vert x_0\Vert} .$$
We conclude that $\frac{x_0}{\Vert x_0\Vert}$ is an APEP of $B_X$ by Theorem~\ref{theo:caraAPEP}, as it is the weak$^*$ limit of a net of extreme points of $B_{X^{**}}$.
\end{proof}

Now we are able to characterise APEPs of norm $1$ for $1<p<\infty$.

\begin{proposition}\label{prop:APEPlpsumas}
Let $X$ be the $\ell_p$-sum of a family $\set{X_i:i\in I}$ of Banach spaces, where $1<p<\infty$, and let $(x_i)\in S_X$. The following assertions are equivalent:
\begin{enumerate}
    \item $(x_i)$ is an APEP of $B_X$. 
    \item For every $i\in I$, either $x_i=0$ or $\frac{x_i}{\Vert x_i\Vert}$ is an APEP of $B_{X_i}$.
\end{enumerate}
\end{proposition}

\begin{proof}
(1)$\Rightarrow$(2): Observe that given $i\in I$ we have 
$$X=X_i\oplus_p Y,$$
where $Y$ is $\ell_p$-sum of the family $\set{X_j:i\in I\setminus\{i\}}$ and the above identification is an isometric isomorphism. Now the result is a direct application of Lemma~\ref{lemma:sumapnececondi}.

(2)$\Rightarrow$(1): Let $U$ be a weak$^*$ neighbourhood of $(x_i)$ in $X^{**}$. We will show that $U$ intersects $\ext{B_{X^{**}}}$ and this will be enough by Theorem~\ref{theo:caraAPEP}. 
Since the weak$^*$ topology of $X^{**}$ is the product topology of the weak$^*$ topologies in $X_i^{**}$, we may assume that $U=\prod_{i\in I} U_i$ where $U_i$ is a weak$^*$ neighbourhood of $x_i$ (or $X_i^{**}$) for each $i\in I$. If $x_i\neq 0$ then, by assumption and Theorem~\ref{theo:caraAPEP}, there exists $e_i\in\ext{B_{X_i^{**}}}$ such that $\norm{x_i}e_i\in U_i$. Let $y=(y_i)\in X^{**}$ be defined by
$$
y_i = \begin{cases}
\norm{x_i}e_i &\text{, if }x_i\neq 0 \\
0 &\text{, if }x_i=0
\end{cases} .
$$
Note that $y\in U$ and
$$
\norm{y} = \pare{\sum_{i\in I}\norm{y_i}^p}^{1/p} = \pare{\sum_{i\in I}\norm{x_i}^p}^{1/p} = \norm{(x_i)} = 1
$$
since $(x_i)\in S_X$, so $y\in \ext{B_{X^{**}}}$ by \eqref{eq:extreme_ell_p}. This ends the proof.
\end{proof}

Of course, Proposition \ref{prop:APEPlpsumas} does not cover all APEPs of the unit ball of $\ell_p$-sums as it is possible to have APEPs of norm strictly less than $1$. For instance, the space $\ell_p$ can be expressed as the $\ell_p$-sum of countably many copies of $\R$, and $\ape{B_{\ell_p}}=B_{\ell_p}$ for $1<p<\infty$ by Proposition \ref{prop:everypointispreserved}.

The analysis of APEPs in an infinite $\ell_\infty$-sum of Banach spaces becomes troublesome because we do not have a simple description of its dual, let alone its bidual. So, instead, we conclude with the case of $c_0$-sums of Banach spaces. In this case, we can characterise APEPs of the unit ball completely.

\begin{proposition}\label{prop:c0sumsAPEP}
Let $X$ be the $c_0$-sum of a family $\set{X_i:i\in I}$ of Banach spaces. Then $x=(x_i)$ is an APEP of $B_X$ if and only if $x_i$ is an APEP of $B_{X_i}$ for each $i\in I$.
\end{proposition}

\begin{proof}
The space $X^{**}$ can be identified with the $\ell_\infty$-sum of the spaces $X_i^{**}$ and, therefore, the extreme points of $B_{X^{**}}$ are precisely the elements of the form $(x_i^{**})$ where $x_i^{**}\in\ext{B_{X_i^{**}}}$ for all $i\in I$. Since the weak$^*$ topology of $X^{**}$ is the product topology of the weak$^*$ topologies in $X_i^{**}$, the result is now clear from Theorem~\ref{theo:caraAPEP}.
\end{proof}

\begin{example}\label{exam:c0(I)}
As consequence of Proposition~\ref{prop:c0sumsAPEP}, the unit ball of $c_0(I)$ does not have any APEP.
\end{example}

\medskip
\section{A characterisation of the RNP in terms of APEP}\label{section:charnp}

In this section we aim to prove a strong connection between the notion of APEP and the RNP. On the one hand, we aim to prove that a Banach space $X$ has the RNP if, and only if, every bounded, closed and convex subset of $X$ has an APEP. On the other hand, we will prove that a Banach space $X$ has the RNP if, and only if, the unit ball of every equivalent renorming of $X$ has an APEP.

Let us start with the first of our objectives. The result we are going to prove is the following.

\begin{theorem}\label{theo:cararnpapepsubconjuntos}
Let $X$ be a Banach space. The following assertions are equivalent:
\begin{enumerate}
    \item $X$ has the RNP.
    \item Every closed, convex and bounded subset $C$ of $X$ has an APEP.
    \item Every closed, convex and bounded subset $C$ of $X$ satisfies
    $$\dist\left(\overline{\ext{\overline{C}^{w^*}}}^{w^*} ,X\right)=0.$$
\end{enumerate}
\end{theorem}

In order to prove the previous result, we need a bit of notation. 
We say that a bounded subset $A \subseteq X^*$ is 
{\it relatively  convexly resolvable with respect to a set} $B \subseteq X^*$ if
there exists a decreasing transfinite sequence of convex and weak* closed sets $(D_\alpha)_{\alpha \leq \kappa}$, 
with $D_\alpha = \bigcap_{\beta <\alpha} D_\beta$ in case $\alpha$ is a limit ordinal, and 
a subset $I \subseteq [1, \kappa]$ such that 
$$ A \subseteq \bigcup_{\alpha \in I} \left( D_{\alpha} \setminus D_{\alpha+1} \right),$$
and 
$$ B \cap \bigcup_{\alpha \in I} \left( D_{\alpha} \setminus D_{\alpha+1} \right) =\emptyset,$$
setting $D_{\kappa+1}=\emptyset$ if necessary.
Obviously, it is compulsory that $A \cap B=\emptyset$. Note that the definition of resolvable set, see \cite[Appendix A.5]{repre}  for instance, can be recovered with 
$B=X^* \setminus A$ and dropping the convexity hypothesis. Typically, we will consider convex resolvability with respect to certain families of weak$^*$ compact subsets of $X^*$.
We say that a subset of a Banach space $A \subseteq X$ is convexly resolvable (relatively with respect to some family of subsets) if it is so as considered in $X^{**}$.\\

Recall that the average range of a vector measure $\nu \colon \Sigma \rightarrow X$ with respect to a positive measure 
$\mu$ on the same measurable space $(\Omega,\Sigma)$ is the set 
$$ \left\{ \frac{\nu(A)}{\mu(A)}: A \in \Sigma, \mu(A) \not = 0 \right\} \subseteq X. $$

\begin{proposition}\label{prop:convexresolimplyrnp}
Let $C \subseteq X$ be a bounded closed convex set. Assume moreover that, either:
\begin{enumerate}
\item $C$ is convexly resolvable with respect to any weak$^*$ compact subset of $X^{**}$ or;
\item $X$ is separable and $C$ is convexly resolvable with respect to any weak$^*$ compact subset $K \subseteq X^{**}$
such that $\dist(K,X)>0$.
\end{enumerate}
Then $C$ has the RNP.
\end{proposition}

The following proof will based on the theory of liftings, see \cite{Tulcea}. Given a measure space $(\Omega,\Sigma,\mu)$ denote by
${\mathcal {L}}^{\infty }(\mu)$ the set of real-valued bounded measurable functions equipped with the essential supremum seminorm and denote by $L^{\infty }(\mu )$ the Banach space obtained by identifying the functions that agree almost everywhere.
A lifting on a measure space $(\Omega,\Sigma,\mu)$ is a linear and multiplicative operator
$ \rho: {\mathcal {L}}^{\infty }(\mu)\to {\mathcal {L}}^{\infty }(\mu)$ such that 
$\rho(f)=\rho(g)$ if and only $f(\omega)=g(\omega)$ for almost all $\omega \in \Omega$. Liftings of 
${\mathcal {L}}^{\infty }(\mu)$ are, somehow, right inverses to the canonical quotient map from 
${\mathcal {L}}^{\infty }(\mu )$ to $L^{\infty }(\mu)$.
The existence of a lifting requires some technical assumptions on the measure space that are fulfilled by the Lebesgue measure, which is enough for our characterization of the RNP.

\begin{proof}
The following argument is essentially developed in \cite{raja}. Without loss of generality we may assume $C \subseteq B_X$.
Firstly, recall an old result of Tortrat, see \cite{Schwartz}, saying that $X$ is universally measurable in $(X^{**}, w^*)$, that is, that $X$ is $\mu$-measurable for every finite Radon measure $\mu$ on $X^{**}$. This implies that $C$ is universally measurable in $(X^{**}, w^*)$  as well.
Let $(\Omega,\Sigma,\mu)$ a probability space and let $\nu \colon \Sigma \rightarrow X$ be a $\mu$-continuous vector measure with
average range in $C$.
We claim that there exists a weak*-Borel measurable density $f \colon \Omega \rightarrow X^{**}$, that
is, $\langle \nu(A), x^* \rangle= \int_{A} \langle f, x^*
\rangle \, d\mu$ for every $A \in \Sigma$ and $x^*\in X^*$, which has a clear barycentric interpretation.

Indeed, for any $x^* \in X^*$, the signed measure $\langle \nu, x^* \rangle$ is
$\mu$-continuous, so it has a Radon-Nikod\'ym derivative
$f_{x^*} \in L^{1}(\mu)$, actually $f_{x^*} \in L^{\infty}(\mu)$ by the boundedness of average ranges.
Let $\rho$ be a lifting of ${\mathcal L}^{\infty}(\mu)$.
It is easy to check
that the map $x^* \to \rho(f_{x^*})(\omega)$ is linear
for every $\omega \in \Omega$ and bounded by $\|x^*\|$, so
there is $x_{\omega}^{**} \in B_{X^{**}}$ such that
$x_{\omega}^{**}(x^*)=\rho(f_{x^*})(\omega)$.
Clearly, the map
defined by $f(\omega)=x_{\omega}^{**}$ is weak*-scalarly
measurable, so it is weak*-Baire measurable by
\cite[Theorem~2.3]{Edgar1}.
Recall that the measure image, also known as the pushforward measure, is defined as $\mu\circ f^{-1}(A)=\mu(f^{-1}(A))$ for $A$ a weak*-Baire set.
The purpose of the measure $\mu\circ f^{-1}$ is to replace the vector valued measure $\nu$ with a scalar measure actually defined on $X^{**}$. The values of $\nu$ can be interpreted as barycenters of restrictions of 
$\mu\circ f^{-1}$.
We claim that $f$ is also weak*-Borel measurable and $\mu\circ f^{-1}$ is weak*-Radon. 

For that, we need to recall the abstract lifting considered in \cite[\S2]{Bellow}. Let $K$ be a compact Hausdorff space and denote by ${\mathcal L}^0(\mu, K)$ the set of $\Sigma$-Baire measurable functions from $\Omega$ to $K$. A  ``scalar'' lifting $\rho$ induces an abstract lifting on ${\mathcal L}^0(\mu, K)$ in the following way: given $\omega \in \Omega$ and 
$f \in {\mathcal L}^0(\mu, K)$, the mapping defined for $h \in C(K)$ by
$$  h \to \rho(h \circ f )(\omega)  $$
is a multiplicative linear functional (character) on $C(K)$. Therefore, exists a unique $x \in K$, depending on $g$ and $\omega$, such that 
$\rho(h \circ f )(\omega) = h(x)$ for all $h \in C(K)$. Thus, we define $\rho_K: {\mathcal L}^0(\mu, K) 
\to {\mathcal L}^0(\mu, K)$ as
$$ \rho_K(f)(\omega)=x .$$
In other words, the abstract lifting satisfies
$$ h \circ \rho_K(f)(\omega)  = \rho(h \circ f )(\omega) $$
for every $h \in C(K)$.
Now we will apply the abstract lifting construction with $K=B_{X^{**}}$, the previous lifting $\rho$ and our function 
$f: \Omega \to B_{X^{**}}$. 
The last equation together with the definition of $f$ gives 
$$ x^* \circ \rho_K(f) (\omega) = \rho(x^* \circ f)(\omega) = \rho( \rho(f_{x^*}))(\omega) = \rho(f_{x^*})(\omega) 
= x^* \circ f (\omega) .$$
for every $x^* \in X^*$.
Therefore $\rho_{K}(f)=f$. By \cite[Theorem~2.1]{Bellow}, $f=\rho_{K}(f)$ is weak*-Borel measurable and $\mu\circ f^{-1}$ is weak*-Radon, as desired.

From this moment on we will work under the assumption (1) of the statement.
Now note that we may assume, without loss of generality, that $f$ takes values in $\overline{C}^{w^*}$.
We will prove that, actually, $f$ has almost all of its values in $C$. 
Indeed, take any weak* compact subset $K \subseteq \overline{C}^{w^*} \setminus C$.
By the hypothesis, there is a
decreasing transfinite sequence of convex and weak* compact sets $(D_\alpha)_{\alpha \leq \kappa}$ and a subset 
$I \subseteq [1,\kappa]$ such that
$$ C \subseteq \bigcup_{\alpha \in I} \left( D_{\alpha} \setminus D_{\alpha+1} \right).$$
In particular, $\overline{C}^{w^*} \subseteq D_1$, thus $K \subseteq D_1$ and so 
$$ K \subseteq \bigcup_{\alpha \not \in I} \left( D_{\alpha} \setminus D_{\alpha+1} \right).$$

{\bf Claim: } $(\mu\circ f^{-1})(K)=0$.\\
Assume that $(\mu\circ f^{-1})(K)>0$. Since $\mu\circ f^{-1}$ is weak*-Radon, there would be a smallest $\alpha$ such that 
$$(\mu\circ f^{-1})(K \cap D_\alpha) = (\mu\circ f^{-1})(K), \mbox{and}$$
$$(\mu\circ f^{-1})(K \cap D_{\alpha+1}) < (\mu\circ f^{-1})(K).~~$$
This implies $(\mu\circ f^{-1}) (K \cap (D_\alpha \setminus D_{\alpha+1})) >0$. In particular, we have\linebreak 
$K \cap (D_\alpha \setminus D_{\alpha+1}) \not = \emptyset$, thus $C \cap (D_\alpha \setminus D_{\alpha+1}) = \emptyset$.
Take $S \subseteq K \cap (D_\alpha \setminus D_{\alpha+1})$ a weak* compact with $(\mu\circ f^{-1}) (S) >0 $ and such that $S$ supports $(\mu\circ f^{-1})|_S$ (Radon measures always have support). By the Hahn-Banach theorem there is a weak* open halfspace $H$ such that $\overline{H}^{w^*} \cap D_{\alpha+1} = \emptyset$ and 
$H \cap S \not = \emptyset$.
Since $S$ is a measure support, $(\mu\circ f^{-1}) (H \cap S) >0$. That would imply together that 
$(\mu\circ f^{-1})(D_\alpha \cap \overline{H}^{w^*})>0$,  and
$D_\alpha \cap \overline{H}^{w^*} \cap C = \emptyset$,
meaning that the average range of $\nu$ lies outside of $C$, which is a contradiction.

Let us assume that assumption (2) holds. The proof will follow the same lines as before with a small variation.
By hypothesis, the bounded closed convex set $C \subseteq X$ is convexly resolvable with respect to any 
$K \subseteq \overline{C}^{w^*} \setminus C$ such that $\dist(K,X)>0$. In particular, the above happens if 
$$ K \subseteq \overline{C}^{w^*} \setminus (X+B_{X^{**}}(0,\varepsilon)) .$$
However, the last set is $w^*$-Borel since 
$$ X+B_{X^{**}}(0,\varepsilon) = \bigcup_{n=1}^\infty B_{X^{**}}(x_n,\varepsilon) $$
where $(x_n) \subseteq X$ is a norm dense sequence. The previous claim shows that any such $K$ satisfies 
$(\mu\circ f^{-1})(K)=0$ and thus the measure is concentrated on
$$  \overline{C}^{w^*} \cap \bigcap_{n=1}^\infty (X+B_{X^{**}}(0,1/n)) = C. $$

Now we know that, under assumptions either (1) or (2), $f$ takes almost all of its values in $C$. Note that $\mu\circ f^{-1}$ is weak*-Radon and therefore weak-Radon in $X$. A classic result attributed by Talagrand to Phillips and Grothendieck \cite{Talagrand, Edgar1} says that 
$\mu\circ f^{-1}$ is the restriction of a Radon measure on $(X, \| \cdot \|)$. That implies that the range of $f$ is essentially separable. Since $f$ is scalarly measurable, we deduce that $f$ is Bochner measurable, and that concludes the proof that $C$ has the RNP.
\end{proof}

\begin{proposition}\label{prop:condisuficonresolva}
Let $C \subseteq X$ be a bounded closed convex set and let $K \subseteq X^{**}$ be a weak* compact. Then either:
\begin{enumerate}
\item $C$ is convexly resolvable with respect to $K$;
\item or there is a nonempty convex closed subset $B \subseteq C$ such that
$$ \ext{\overline{B}^{w^*}} \subseteq K .$$
\end{enumerate}
\end{proposition}

\begin{proof}
Assume that $ \ext{\overline{B}^{w^*}} \setminus K \not = \emptyset$ for every nonempty convex closed subset 
$B \subseteq C$. We will inductively build a sequence $(D_\alpha)_{\alpha \geq 1}$ witnessing the convex solvability of $C$ with respect to $K$. For that, we will also build a sequence of nonempty closed convex sets 
$(B_\alpha)_{\alpha \mbox{\tiny{\ odd}}}$ and a sequence of $w^*$-open halfspaces $(H_\alpha)_{\alpha \mbox{\tiny{\ odd}}}$ in $X^{**}$. 
Take $B_1 =C$. Assume that $\alpha$ is an odd ordinal such that $B_\alpha \not = \emptyset$ and set $D_\alpha=
\overline{B_\alpha}^{w^*}$. By our assumption, there is some $x \in \ext{D_\alpha} \setminus K$, we may take a weak* open halfspace $H_\alpha$ such that $x \in H_\alpha$ and $D_\alpha \cap H_\alpha \cap K =\emptyset$.
Then take $D_{\alpha+1}=D_\alpha \setminus H_\alpha$ and $B_{\alpha+2}= B_\alpha \setminus H_\alpha$.
For $\alpha$ a limit ordinal just take $D_\alpha = \bigcap_{\beta<\alpha} D_\beta$, that is nonempty if the previous sets were so, and $B_{\alpha+1}=D_\alpha \cap C$ (limits ordinals are even).
Continue the process while $B_\alpha \not = \emptyset$ and eventually $C$ will be exhausted. Therefore
$$ C \subseteq \bigcup_{\alpha \mbox{\tiny{\ is odd}}} (B_\alpha \cap H_\alpha) 
\subseteq \bigcup_{\alpha \mbox{\tiny{\ is odd}}} (D_\alpha \setminus D_{\alpha+1})$$
and this last set does not meet $K$ by construction. That proves the convex solvability of $C$ with respect to $K$.
\end{proof}

Now we can provide the pending proof.

\begin{proof}[Proof of Theorem~\ref{theo:cararnpapepsubconjuntos}]
(1)$\Rightarrow$(2): Every closed, convex and bounded subset $C$ of $X$ has a denting point since $X$ has the RNP, and denting points are APEPs.
(2)$\Rightarrow$(3):  By Theorem \ref{theo:caraAPEP} we have
$$ \emptyset \not = \ape{C} \subseteq X \cap \overline{\ext{\overline{C}^{w^*}}}^{w*} .$$
(3)$\Rightarrow$(1): Assume that $X$ does not have the RNP. Since the RNP is separably determined, there exists a separable subspace $Y \subseteq X$ lacking the RNP, and so its unit ball $B_Y$ lacks the RNP. By 
Proposition \ref{prop:convexresolimplyrnp}, there exists a $w^*$-compact $K \subseteq Y^{**}$ with $\dist(Y,K)>0$ such that  
$B_Y$ is not convexly resolvable with respect to $K$. Now, by Proposition~\ref{prop:condisuficonresolva}, there exists a nonempty convex closed set $B \subseteq B_Y$ such that 
$$ \ext{\overline{B}^{w^*}} \subseteq K ,$$
where the distance is computed inside $Y^{**}$. Having in mind the natural isometric embedding of $Y^{**}$ into $X^{**}$ and applying the fact that $$\dist(y^{**}, X) \geq (1/2) \dist(y^{**},Y)$$ 
for any $y^{**} \in Y^{**}$, see \cite[Lemma 2.3]{scsewe89} or \cite[Proposition 3.59]{BIO} for instance, we deduce that  
$$ \dist\left ( \overline{\ext{\overline{B}^{w^*}}}^{w^*} , X \right)\geq
\frac{1}{2}\dist\left ( \overline{\ext{\overline{B}^{w^*}}}^{w^*} , Y \right)>0 $$
as wished.
\end{proof}

Once we have accomplished one of the main aims of the section, we want to take advantage of Theorem~\ref{theo:cararnpapepsubconjuntos} to prove that, in any Banach space $X$ that fails the RNP, there exists an equivalent renorming whose unit ball has no APEP. That is what we do in the following theorem.

\begin{theorem}\label{theo:renormasinapepnornp}
Let $X$ be a Banach space failing the RNP. Then there exists an equivalent norm $\vert\cdot\vert$ on $X$ such that $\ape{B_{(X,\vert\cdot\vert)}}=\emptyset$.
\end{theorem}

\begin{proof}
Since $X$ fails the RNP, by Theorem~\ref{theo:cararnpapepsubconjuntos} there exists a closed, convex and bounded subset $C\subseteq X$ and a $w^*$-compact subset $K \subseteq X^{**}$ with $\dist(K,X) =\eta>0$ such that
$$ \ext{\overline{C}^{w^*}} \subseteq K .$$
Take $D = \cco(C \cup (-C))$ and note that 
$$ \ext{\overline{D}^{w^*}} \subseteq \ext{\overline{C}^{w^*}} \cup \left(- \ext{\overline{C}^{w^*}} \right) 
\subseteq H:= K \cup (-K) .$$
Note that $\dist(H,X)=\eta$ and thus $\dist(H+\delta B_{X^{**}},X)=\eta-\delta$ provided that $\delta \leq \eta$. 
Take $B=\cco(D+B_X(0,\eta/2))$. We have
$$  \ext{\overline{B}^{w^*}} \subseteq  \ext{\overline{D}^{w^*}} + (\eta/2)B_{X^{**}} \subseteq H+ (\eta/2) B_{X^{**}},$$
and so
$$ \dist\left (\ext{\overline{B}^{w^*}} , X\right) \geq \eta/2 >0. $$
That implies that $\ape{B}=\emptyset$ by Theorem~\ref{theo:caraAPEP}. Consequently, the equivalent norm whose unit ball is $B$ does the trick.
\end{proof}

As a consequence of Theorem~\ref{theo:renormasinapepnornp} we are now able to obtain the following characterisation of the RNP.

\begin{corollary}\label{cor:cararnpequireno}
Let $X$ be a Banach space. The following assertions are equivalent:
\begin{enumerate}
    \item $X$ has the RNP.
    \item The unit ball of any equivalent renorming of $X$ has an APEP.
\end{enumerate}
\end{corollary}

\section{Lipschitz-free spaces}\label{section:Lipschitzfree}

In this section, we will study APEPs of the unit ball of Lipschitz-free spaces and relate them to extremal structure. We refer the reader to Section~\ref{subsection:Lipschitz-free} for notation and basic facts. Throughout the section, $(M,d)$ will denote a complete metric space with base point $0\in M$.

We begin by vastly reducing our list of suspects:

\begin{proposition}\label{pr:apep_molecule_0}
Every APEP of $B_{\F{M}}$ is either a molecule or $0$.
\end{proposition}

\begin{proof}
Since $B_{\F{M}}=\cconv(\Mol{M})$, Lemma \ref{lemma:milman_apep} implies that any APEP of $B_{\F{M}}$ must belong to $\wcl{\Mol{M}}$. However, $\wcl{\Mol{M}}\subseteq\Mol{M}\cup\set{0}$ by \cite[Proposition 2.9]{gppr18}.
\end{proof}

By \cite[Lemma 4.2]{GRZ}, $\Mol{M}$ is weakly closed if and only if $M$ bi-Lipschitz embeds into $\R^n$ for some $n$. If that is the case, the argument above shows that any APEP of $B_{\F{M}}$ must be a molecule. If $M$ does not bi-Lipschitz embed into Euclidean space, then $0$ can be or fail to be an APEP of $B_{\F{M}}$, as witnessed by the following examples.

\begin{example}\label{example:0apepfree}
Let $X$ be any infinite-dimensional uniformly convex Banach space and $M:=S_X$. Then $M$ is uniformly concave and therefore every molecule in $\F{M}$ is a preserved extreme point of $B_{\F{M}}$ \cite[Theorem~3.39]{Weaver}. Clearly, $M$ does not bi-Lipschitz embed into Euclidean space, as it is not even locally compact. Thus $0$ is an APEP of $B_{\F{M}}$ by Theorem \ref{theo:caraAPEP} and \cite[Lemma 4.2]{GRZ}. 
\end{example}

\begin{example}\label{example:no0apepfree}
Let $M:=\set{0}\cup\set{e_n:n\in\N}\subseteq\ell_1$. Clearly, $M$ does not bi-Lipschitz embed into Euclidean space as it is not totally bounded. However, $\F{M}$ is linearly isometric to $\ell_1$ (see e.g. \cite[Example 3.10]{Weaver}) and therefore $0$ is not an APEP of $B_{\F{M}}$ by Proposition \ref{prop:l1purelyatomic}.
\end{example}

We do not know of a precise metric condition characterising when $0$ is an APEP of $B_{\F{M}}$. But we are able to characterise those molecules that are APEPs as follows.

\begin{theorem}\label{th:apep_molecules}
Let $m_{xy}\in\F{M}$ be a molecule. Then the following are equivalent:
\begin{enumerate}
\item $m_{xy}\in\ape{B_{\F{M}}}$,
\item $m_{xy}\in\cl{\dent{B_{\F{M}}}}$,
\item there exist $m_{x_ny_n}\in\dent{B_{\F{M}}}$ such that $x_n\to x$ and $y_n\to y$.
\end{enumerate}
\end{theorem}

\begin{proof}
The equivalence (2)$\Leftrightarrow$(3) follows from the fact that all denting points of $B_{\F{M}}$ are molecules and that norm convergence of molecules translates to convergence of the underlying pair of points. On the other hand, (2) is equivalent to $m_{xy}\in\wcl{\dent{B_{\F{M}}}}$ by \cite[Lemma 2.2]{gppr18}, and this clearly implies (1) by Lemma \ref{lemma:apep_w_closed}. So it only remains to be proved that (1) implies $m_{xy}\in\wcl{\dent{B_{\F{M}}}}$.

Let $m_{xy}$ be an APEP of $B_{\F{M}}$ and let $W$ be a weak neighbourhood of $m_{xy}$. We will show that $W$ contains a denting point of $B_{\F{M}}$ and this will finish the proof. Apply Lemma \ref{lemma:apepsmallslices} to $m_{xy}$, with $B=\Mol{M}$. Since norm and weak convergence of molecules agree, the situation in option 1) of Lemma \ref{lemma:apepsmallslices} is impossible, so we deduce that $W$ contains slices of arbitrarily small diameter.  It follows from Remark~\ref{remark:slicelocaldiam} that $W$ contains a slice $S=S(B_{\F{M}},f,\alpha)$ where $f$ is not local. By \cite[Proposition 2.7]{vee25} $S$, and thus $W$, contains a denting point of $B_{\F{M}}$. This ends the proof.
\end{proof}

We will now use Theorem \ref{th:apep_molecules} to provide further examples of the interplay between APEPs and extremality. In Section \ref{section:firstresults}, we showed that APEPs are not necessarily extreme points. This can also happen to molecules in Lipschitz-free spaces.

\begin{example}\label{ex:lipfree_apep_not_extreme}
Let $M = \set{0,1}\times [0,1] \cup \set{(\tfrac{1}{2},0)} \subseteq \R^2$.
Then each $m_n:=m_{(0,\frac{1}{n}),(1,\frac{1}{n})}$ is an extreme point of $B_{\F{M}}$; in fact, it is preserved because $M$ is compact \cite[Theorem 4.2]{AliagaGuirao}. Moreover, $m_n$ converges to $m:=m_{(0,0),(1,0)}$, so $m$ is APEP by Theorem \ref{th:apep_molecules}. But $m$ is not extreme because the segment between $(0,0)$ and $(1,0)$ is not trivial.
\end{example}

Observe that, up to this point, we have not obtained any example of a (non-preserved) extreme point which fails to be an APEP. The next example shows that this is indeed possible.

\begin{example}\label{exam:extremenoapepfree}
Let $M$ be a uniformly discrete metric space such that there exists an extreme point $m_{xy}\in\ext{B_{\F{M}}}$ that is not preserved (such examples exist, e.g. \cite[Example 4.3]{AliagaGuirao}). We claim that $m_{xy}$ is not an APEP.
Indeed, by Theorem \ref{th:apep_molecules}, $m_{xy}$ can only be an APEP if there exist preserved extreme points $m_{x_ny_n}$ such that $x_n\to x$ and $y_n\to y$ but, since $M$ is topologically discrete, this implies that $x_n=x$, $y_n=y$ for $n$ large enough, and thus $m_{xy}$ is preserved.
\end{example}

It is also possible to construct a metric space $M$ such that $B_{\F{M}}$ has a non-preserved extreme point that is the norm limit of preserved extreme points, hence APEP by Theorem \ref{th:apep_molecules}:

\begin{example}\label{example:extreapepnopreser}
In $\R^2$ with the $\ell_1$ metric, consider the set $N$ consisting of the points $p=(0,0)$, $q=(1,0)$, and $p_k=(0,\frac{1}{k})$, $q_k=(1,\frac{1}{k})$ for $k\in\N$. Let $M=N\cup\set{x_n:n\in\N}$, endowed with the aforementioned metric for $N$ and
\begin{align*}
d(x_n,p) = d(x_n,q) &= \tfrac{1}{2}+\tfrac{1}{n} \\
d(x_n,p_k) = d(x_n,q_k) &= \tfrac{1}{2}+\tfrac{1}{k}+\tfrac{1}{n} \\
d(x_n,x_m) &= 1+\tfrac{1}{n}+\tfrac{1}{m}
\end{align*}
for $n\neq m\in\N$. Note that the quantity
$$
d(p_k,x) + d(q_k,x) - d(p_k,q_k)
$$
is at least $2\pare{\frac{1}{k}-\frac{1}{k+1}}$ for all $k\in\N$ and $x\in M\setminus\set{p_k,q_k}$, therefore every molecule $m_{p_kq_k}$ is a preserved extreme point of $B_{\F{M}}$. Since $p_k\to p$, $q_k\to q$, the molecule $m_{pq}$ is an APEP by Theorem \ref{th:apep_molecules}. Similarly, the quantity
$$
d(p,x) + d(q,x) - d(p,q)
$$
is strictly positive for $x\in M\setminus\set{p,q}$, however its value for $x=x_n$ is $\frac{2}{n}$, which can be made arbitrarily small while $d(p,x),d(q,x)>\frac{1}{2}$. Thus $m_{pq}$ is an unpreserved extreme point of $B_{\F{M}}$.
\end{example}

Observe that Theorem~\ref{th:apep_molecules} reveals that, for a molecule $m_{xy}$ in $\mathcal F(M)$, if $m_{xy}$ is an APEP then every weakly open subset $W$ of $B_{\mathcal F(M)}$ containing $m_{xy}$ contains slices of arbitrarily small diameter. At this point, it could be wondered whether, in the particular case of Lipschitz-free spaces, this phenomenon occurs because, in fact, $m_{xy}$ is contained in non-empty weakly open subsets of arbitrarily small diameter. In the following example we show that this does not hold and that it is possible for a molecule $m_{xy}$ to be an APEP and have the property that every non-empty weakly open subset $W$ of $B_{\mathcal F(M)}$ containing $m_{xy}$ satisfies $\diam(W)=2$. 

Before exhibiting the example, let us recall that a point $x$ in the unit sphere of a Banach space $X$ is said to be a \textit{$\Delta$-point} (respectively \textit{super $\Delta$-point}) if every slice (resp. relatively weakly open subset) of the unit ball of $X$ containing $x$ contains points at distance $2-\varepsilon$ from $x$ for every $\varepsilon>0$.

\begin{example}\label{example:veeorgapep}
Consider the metric space $M_V$ constructed by Veeorg in \cite[Section 3]{Veeorg} as follows. In $\R^2$, consider the points $p=(0,0)$ and $q=(1,0)$ and, for $n\in\N$, the set
$$
S_n = \set{(2^{-n}k,2^{-n}) : k=0,1,\ldots,2^n} .
$$
Then $M_V$ is the set $\set{p,q}\cup\bigcup_{n=1}^\infty S_n$ endowed with the metric
$$
d((x_1,y_1),(x_2,y_2)) = \begin{cases}
\abs{x_1-x_2} &\text{, if }y_1=y_2 \\
\abs{y_1-y_2}+\min\set{x_1+x_2,2-(x_1+x_2)} &\text{, if }y_1\neq y_2
\end{cases} .
$$
Now let $p_n=(0,\frac{1}{2^n+1})$, $q_n=(1,\frac{1}{2^n+1})$ and set $M=M_V\cup\set{p_n,q_n:n\in\N}$ with the metric defined by the same formula. See Figure \ref{figure:veeorgapep}.

\begin{figure}[ht]
  \centering
  \begin{tikzpicture}[scale=6.0]
    \draw[dotted] (-0.01,0) -- (1.0,0);
    \draw (0,-0.01) -- (0,0.5);
    \draw (1,0) -- (1,1/2);

    \foreach \y in {1/16, 1/8, 1/4, 1/2} {
      \draw (0.01,\y) -- (-0.01,\y) node[left] {$\y$};
    }

    \fill (0,0) circle[radius=0.4pt] node[below left] {$p$};
    \fill (1,0) circle[radius=0.4pt] node[below right] {$q$};

    \foreach \sn in {1,2,3,4} {
      \draw (0,{2^(-\sn)}) -- (1,{2^(-\sn)});
      \pgfmathtruncatemacro\aa{2^\sn}
      \foreach \k in {0,...,\aa} {
        \fill ({\k*2^(-\sn)},{2^(-\sn)}) circle[radius=0.4pt];
      }
    }

    \foreach \sn in {1,2,3} {
      \fill (0,{1/(2^(\sn))*3/4}) circle[radius=0.4pt] node[right] {$p_\sn$};
      \fill (1,{1/(2^(\sn))*3/4}) circle[radius=0.4pt] node[right] {$q_\sn$};
      \draw[dotted] (-0.01,{1/(2^(\sn))*3/4}) -- (1.0,{1/(2^(\sn))*3/4});
    }
  \end{tikzpicture}
  \caption{The metric space $M$ from Example \ref{example:veeorgapep}.}
  \label{figure:veeorgapep}
\end{figure}
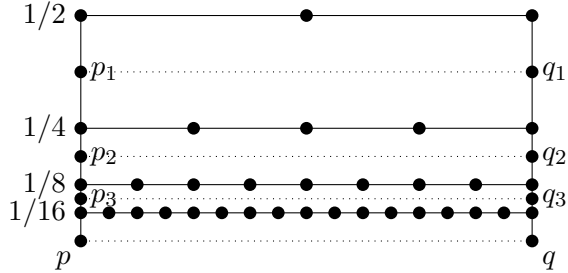

Similarly to Example \ref{example:extreapepnopreser}, the molecule $m_{pq}$ is an unpreserved extreme point of $B_{\F{M}}$ that is an APEP. Indeed, any $u\in M\setminus\set{p,q}$ has the form $u=(x,y)$ for some $y>0$ and therefore
$$
d(p,u)+d(q,u)-d(p,q) = (x+y)+(1-x+y)-1 = 2y
$$
is always positive, but can be made arbitrarily small while keeping e.g. $x=\frac{1}{2}$ to make sure that $d(p,u)$ and $d(q,u)$ remain large, so $m_{pq}$ is an unpreserved extreme point. Moreover, each molecule $m_{p_nq_n}$ is a denting point of $B_{\F{M}}$ as, given any $u=(x,y)\in M\setminus\set{p_n,q_n}$, a similar computation yields
$$
d(p_n,u)+d(q_n,u)-d(p_n,q_n) = 2\abs{y-\frac{1}{2^n+1}}
$$
which has a positive lower bound in $M\setminus\set{p_n,q_n}$. Thus $m_{pq}$ is an APEP of $B_{\F{M}}$ by Theorem \ref{th:apep_molecules} as $p_n\to p$, $q_n\to q$.

In this example, $m_{pq}$ is even a $\Delta$-point. This follows from \cite[Theorem 6.7]{AALMPPV1} as $p$ and $q$ are \textit{discretely connectable} in $M$, i.e. their distance can be approximated by discrete paths in $M$ with arbitrarily small jumps (passing through the sets $S_n$).
\end{example}

\begin{remark}
Recently, E. Basset, Y. Perreau, A. Proch\'azka and T. Veeorg announced the result that, in any Lipschitz-free space, every molecule that is a $\Delta$-point is a super $\Delta$-point. This (still unpublished) result would imply that the APEP $m_{pq}$ from Example \ref{example:veeorgapep} is a super-$\Delta$ point. In particular, APEPs can fail to be points of (weak- to norm-)continuity even in Banach spaces with the RNP.

An even stronger notion is that of \textit{Daugavet point}, i.e. a point $x\in S_X$ such that every slice $S$ of $B_X$ contains points at distance $2-\varepsilon$ from $x$, regardless of whether $x\in S$ or not. However, for Lipschitz-free spaces, APEPs of $B_{\F{M}}$ can never be Daugavet points, as Daugavet points are at distance $2$ from any denting point by \cite[Proposition 3.1]{JRZ}.
\end{remark}

\medskip
\section{Tensor products}\label{section:tensor}

In this section, we aim to study APEPs in projective tensor product spaces. Following the spirit of previous works dealing with the extremal structure in projective tensor products like \cite{ggmr23,werner87}, we will focus on studying the APEPs of sets of the form $\cconv(C\otimes D)$ in $X\pten Y$ for bounded, closed and convex subsets $C\subseteq X$ and $D\subseteq Y$. We refer the reader to Subsection~\ref{subsection:tensor} for necessary notation and background on tensor product theory.

Let us start with the search for necessary conditions for APEPs in $\cconv(C\otimes D)$. It is natural that they would have to be elementary tensors. This is precisely the statement of the next result under appropriate assumptions on the space $K(X,Y^*)$.

\begin{theorem}\label{theo:tensornece}
    Let $X$ and $Y$ be Banach spaces such that $K(X,Y^*)$ is separating for $X \pten Y$. Let $C\subseteq X$ and $D\subseteq Y$ be bounded, closed and convex subsets. If $z$ is an APEP of $\cconv(C\otimes D) \subseteq X\pten Y$, then $z=x\otimes y$ for some $x\in C$ and $y\in D$.
\end{theorem}

This result should be compared with \cite[Theorem 1.1]{ggmr23}, where a similar statement is proved for preserved extreme points.

\begin{proof}
Let $z$ be an APEP of $\cconv(C\otimes D)$. An application of Lemma~\ref{lemma:milman_apep} yields that
$$z\in \overline{C\otimes D}^w = \overline{C}^w \otimes \overline{D}^w,$$
thanks to \cite[Theorem 2.3]{ggmr23}. Finally, since $C,D$ are weakly closed, it follows that $z=x\otimes y$ for some $x\in C$ and $y\in D$.\end{proof}

\begin{remark}\label{remark:teonecetensor}
In view of \cite[Theorem 1.1]{ggmr23} it is natural to suspect that, in the above theorem, if $z=x\otimes y\neq 0$ is an APEP then both $x$ and $y$ should be APEPs in $C$ and $D$ respectively. However, we will show in Example~\ref{example:tensorAPEPfactorno} that such a result does not hold.
\end{remark}

In order to establish sufficient conditions for APEPs of a set of the form $\cconv(C\otimes D)$, the above result says that we can reduce to analysing the elementary tensors. Now in our first sufficient condition we center our attention on \cite[Proposition 3.2]{ggmr23}, where it is proved that if $x_0$ is a strongly exposed point of $C$ and $y_0$ is a weakly strongly exposed point of $D$ then $x_0\otimes y_0$ is a weakly strongly exposed point. However, since in the notion of APEP we do not need to localise the point $x_0\otimes y_0$, it seems that the assumptions that $x_0$ is denting and $y_0$ is APEP should be enough to get that $x_0\otimes y_0$ is APEP. This is precisely the content of the following result.

\begin{theorem}\label{theo:tensordienteAPEP}
    Let $X,Y$ be Banach spaces. Let $C\subseteq X$ and $D\subseteq Y$ be bounded, closed,  convex, and symmetric subsets. Let $x_0$ be a denting point of $C$ and $y_0$ be an APEP of $D$. Then, $x_0\otimes y_0$ is an APEP of $\cconv(C\otimes D)$.
\end{theorem}

\begin{proof}
    We can assume that $C, D$ are both different from $\{0\}$, since otherwise $\cconv(C\otimes D)=\{0\}$. Hence, without loss of generality, assume that $\sup_{z\in C}\Vert z\Vert=\sup_{w\in D}\Vert w\Vert=1$.
    Let $U$ be a relatively weakly open neighbourhood of $x_0\otimes y_0$ in $\cconv(C\otimes D)$. We may assume that $U=\bigcap_{i=1}^n S(\cconv(C\otimes D), T_i, \alpha_i)$ for some $T_1,\ldots, T_n \in S_{L(X,Y^*)}$ and some $\alpha_1,\ldots, \alpha_n >0$. Since $x_0\otimes y_0\in U$, we have $T_i(x_0)(y_0)> \sup_{z\in \cconv(C\otimes D)} T_i(z)-\alpha_i$ for every $1\leq i\leq n$. Thus, we may find $\varepsilon_0>0$ so that $T_i(x_0)(y_0)> \sup_{z\in \cconv(C\otimes D)} T_i(z)-\alpha_i+\varepsilon_0$ for every $1\leq i\leq n$. Since $x_0 \in \dent{C}$, there are some $\delta'>0$ and $x^*\in X^*$ with $\sup_{x\in C} x^*(x)=1$ such that $\diam(S(C,x^*,\delta' ))<\frac{\varepsilon_0}{4} $. Moreover, notice that $y_0$ belongs to the set
    $$W = \bigcap_{i=1}^n \left\{ y\in D : T_i(x_0)(y)> \sup_{z\in \cconv(C\otimes D)}T_i(z)-\alpha_i+\varepsilon_0 \right\}$$
    which is a relatively weakly open subset of $D$. Since $y_0$ is an APEP of $D$, there are some $\delta''>0$ and $y^*\in Y^*$ with $\sup_{y\in D}y^*(y)=1$ such that 
    $S\left(D,y^*,\delta''\right) \subseteq W$.
    Finally, taking $\delta=\min\left\{ \delta', \delta'', \frac{\varepsilon_0}{4}\right\}$ and considering the slice $S=S(\cconv(C\otimes D), x^*\otimes y^*, \eta^2)$ where $0<\eta<\min\left\{ \frac{\delta}{4}, \frac{1}{2}\right\}$, we conclude that $S\subseteq U$, exactly as in the proof of \cite[Proposition 3.2]{ggmr23}. Thus, $x_0\otimes y_0$ is an APEP of $\cconv(C\otimes D)$ as desired.
\end{proof}

Now we continue looking for sufficient conditions for a point $x_0\otimes y_0$ to be an APEP. Let us consider the following definition from \cite{ggmr23}.

\begin{definition}\label{defi:compactneigh}
Let $X$ and $Y$ be Banach spaces, and let $A\subseteq X\pten Y$. We say that $u\in A$ \emph{has a compact neighbourhood system for the weak topology in $A$} if, given any weakly open subset $U$ containing $u$, there are slices $S(A, T_i,\alpha_i)$ given by compact operators $T_i\in K(X, Y^*)$ such that 
\[ u \in \bigcap\limits_{i=1}^n S(A, T_i,\alpha_i)\subseteq U.\]     
\end{definition}

\begin{remark}\label{remark:equideficompactneigh}
It is immediate that the above definition is an equivalent reformulation to the definition given in \cite[Definition 3.3]{ggmr23}. It is also clear that $u$ has a compact neighbourhood system for the weak topology in $A\subseteq X\pten Y$ if and only if $u$ is a point of continuity of the identity map 
$$id\colon (A,\sigma(X\pten Y, K(X,Y^*)))\to (A,w).$$
\end{remark}

It is now time for some examples of this situation.

\begin{example}\label{example:compactneigh}\hfill
\begin{enumerate}
\item Given two Banach spaces $X,Y$ such that $L(X,Y^*)=K(X,Y^*)$, it is clear that every $u\in A$ has a compact neighbourhood system for the weak topology in $A$ for every $A\subseteq X\pten Y$.
\item Let $C\subseteq X$ and $D\subseteq Y$ be two closed, absolutely convex and bounded subsets and let $x_0\in \dent{C}, y_0\in \dent{D}$. Then $x_0\otimes y_0$ has a compact neighbourhood system for the weak topology in $\cconv(C\otimes D)$. Indeed, it is proved in \cite[Theorem 1]{werner87} that given $\varepsilon>0$ there exist $x_0^*\in X^*$ and $y_0^*\in Y^*$ such that $x_0\otimes y_0\in S(\cconv(C\otimes D), x_0^*\otimes y_0^*,\alpha)$ and that $\diam(S(\cconv(C\otimes D), x_0^*\otimes y_0^*,\alpha))<\varepsilon$. The result follows since the operator
$$\begin{array}{cccc}
   x_0^*\otimes y_0^*\colon & X  & \longrightarrow & Y^*  \\
    & x & \longmapsto & x_0^*(x)y_0^*
\end{array}$$
is clearly compact.

\item In \cite[Example 3.8]{ggmr23} an equivalent norm $\vert\cdot\vert$ on $\ell_2$ is given such that $e_1$ is a weakly strongly exposed point of $B_{(\ell_2,\vert\cdot\vert)}$ but $e_1\otimes e_1$ is not a weakly strongly exposed point of $B_{(\ell_2,\vert\cdot\vert)\pten (\ell_2,\vert\cdot\vert)}$. According to \cite[Theorem 1.3]{ggmr23}, the point $e_1\otimes e_1$ fails to have a compact neighbourhood system for the weak topology in $B_{(\ell_2,\vert\cdot\vert)\pten (\ell_2,\vert\cdot\vert)}$.
\end{enumerate}
\end{example}

The following result establishes that APEPs remain stable under tensor products, provided that a suitable compact neighbourhood system exists.

\begin{theorem}\label{theo:tensorAPEPcompactneigh}
Let $X$ and $Y$ be two Banach spaces, and let $C\subseteq X$, $D\subseteq Y$ be bounded, closed,  convex and symmetric subsets. Let $x_0\in\ape{C}$ and $y_0\in\ape{D}$. Assume that $x_0\otimes y_0$ has a compact neighbourhood system for the weak topology in $\cco(C\otimes D)\subseteq X\pten Y$. Then $x_0\otimes y_0$ is an APEP of $\cco(C\otimes D)$.
\end{theorem}

This result should be compared with \cite[Theorem 1.3]{ggmr23}, where an analogous statement is proved for weakly strongly exposed points. Moreover, we employ here many ideas from the proof of that result.

\begin{proof} We can assume that $C, D$ are both different from $\{0\}$, since otherwise $\cconv(C\otimes D)=\{0\}$. Hence, without loss of generality, assume that $\sup_{z\in C}\Vert z\Vert=\sup_{w\in D}\Vert w\Vert=1$.
Let $U$ be a weak neighbourhood of $x_0\otimes y_0$ in $\cconv(C\otimes D)$. By the assumption, we can assume that $U=\bigcap\limits_{i=1}^n S(\cconv(C\otimes D), T_i,\alpha_i)$ for certain compact operators $T_1,\ldots, T_n\colon X\to Y^*$. Furthermore, we can assume $\sup_{u\in \cconv(C\otimes D)}T_i(u)=1$ for every $i$. Let $\eta>0$ small enough so that $T_i(x_0\otimes y_0)>1-\alpha_i+\eta$ holds for every $1\leq i\leq n$. Moreover, observe that $x_0\in \bigcap\limits_{i=1}^n \{z\in C: T_i(z)(y_0)>1-\alpha_i+\eta\}$, which is a relatively weakly open subset of $C$. Since $x_0$ is an APEP of $C$ there exists a slice $S(C,x_0^*,\delta')$ 
such that $\sup_{z\in C} x_0^*(z)=1$ and that \[S(C, x_0^*,\delta')\subseteq \bigcap\limits_{i=1}^n \{z\in C: T_{i}(z)(y_0)>1-\alpha_i+\eta\}.\]

Now, for every $1\leq i\leq n$, the set $T_i(S(C, x_0^*,\delta'))$ is a relatively compact subset of $Y^*$. Using the compactness condition on all the $T_i$ we can find a finite set $x_1,\ldots, x_m\in S(C, x_0^*,\delta')$ so that the balls $B(T_i(x_j),\frac{\eta}{2}), 1\leq j\leq m$, cover $T_i(S(C, x_0^*,\delta'))$ for every $1\leq i\leq n$. Observe that $T_i(x_j)(y_0)>1-\alpha_i+\eta$ holds for every $1\leq i\leq n$ and $1\leq j\leq m$. Consequently, 
$$y_0\in\bigcap\limits_{i=1}^n\bigcap\limits_{j=1}^m\{y\in D: T_i(x_j)(y)>1-\alpha_i+\eta\}.$$
Since $y_0$ is an APEP of $D$ we can find a slice $S(D,y_0^*,\delta'')$ such that $\sup_{w\in D}y_0^*(w)=1$ and that \[S(D,y_0^*,\delta'')\subseteq  \bigcap\limits_{i=1}^n\bigcap\limits_{j=1}^m\{y\in D: T_i(x_j)(y)>1-\alpha_i+\eta\}.\]

We claim now that 
\[ S(C, x^*_0, \delta')\otimes S(D, y^*_0, \delta'')\subseteq \bigcap_{i=1}^n S\left(\cconv(C\otimes D), T_i, \alpha_i-\frac{\eta}{2}\right).\]
Indeed, let $x\in S(C, x^*_0,\delta')$ and $y\in S(D, y^*_0, \delta'')$. We have, for every $i\in\{1,\ldots, n\}$, an index $j_i\in \{1,\ldots, m\}$ such that $\Vert T_i(x)-T_i(x_{j_i})\Vert<\frac{\eta}{2}$. On the other hand, since $S(D,y_0^*,\delta'')\subseteq  \bigcap\limits_{i=1}^n\bigcap\limits_{j=1}^m\{y\in D: T_i(x_j)(y)>1-\alpha_i+\eta\}$ we have that, for every $1\leq i\leq n$, $T_i(x_{j_i})(y)>1-\alpha_i+\eta$. Consequently
\[
T_i(x)(y)\geq T_i(x_{j_i})(y)-\Vert T_i(x_{j_i})-T_i(x)\Vert >1-\alpha_i+\eta-\frac{\eta}{2}
 =1-\alpha_i+\frac{\eta}{2}.
\]

Take $\delta:=\min\{\delta',\delta'', \frac{\eta}{8\max_{1\leq i\leq n}{\norm{T_i}}},\frac{1}{2}\}$ and consider $S:=S(\cconv(C\otimes D),x_0^*\otimes y_0^*, \delta^2)$ which is non-empty since $\sup_{z\in C} x_0^*(z)=1=\sup_{w\in D}y_0^*(w)$.  Moreover, 
$$S\subseteq \co(S(C\otimes D,x_0^*\otimes y_0^*,\delta))+4\delta B_{X\pten Y}$$
by virtue of \cite[Lemma 2.1]{ggmr23}. Now, given $1\leq i\leq n$, since $1-\delta\geq\max\{1-\delta',1-\delta''\}$ we conclude that every element  $x\otimes y$ of $S(C\otimes D,x_0^*\otimes y_0^*,\delta)$ satisfies  $x_0^*(x)>1-\delta'$ and $y_0^*(y)>1-\delta''$, so $T_i(x)(y)>1-\alpha_i+\frac{\eta}{2}$. Since $T_i$ is a linear continuous functional on $X\pten Y$ we conclude that $T_i(z)\geq 1-\alpha_i+\frac{\eta}{2}$ holds for every $1\leq i\leq n$ and every $z\in \co(S(C\otimes D,x_0^*\otimes y_0^*,\delta))$. Henceforth, given $z\in S$ we can find $u\in \co(S(C\otimes D,x_0^*\otimes y_0^*,\delta))$ and $v\in B_{X\pten Y}$ so that $z=u+4\delta v$. Now, given $1\leq i\leq n$ we get
$$T_i(z)=T_i(u)+4 \delta T_i(v)\geq 1-\alpha_i+\frac{\eta}{2}-4\delta\norm{T_i}>1-\alpha_i,$$
from which we conclude that $z\in \bigcap\limits_{i=1}^n S(\cconv(C\otimes D), T_i,\alpha_i)=U$. This proves that $S\subseteq U$.

Summarising, we have proved that every relatively weakly open subset of $\cconv(C\otimes D)$ containing $x_0\otimes y_0$ actually contains a slice $S(\cconv(C\otimes D), x_0^*\otimes y_0^*,\alpha)$. From here we conclude that $x_0\otimes y_0$ is an APEP, as requested.
\end{proof}

Let us now observe that we do not need both elements $x_0$ and $y_0$ to be APEPs in Theorem~\ref{theo:tensorAPEPcompactneigh}.

\begin{example}\label{example:tensorAPEPfactorno}
Let $X=C([0,1])$ and let $Y=\ell_p$ for $2<p<\infty$. Clearly both $X$ and $Y$ have the AP and, moreover,
$$L(X,Y^*)=L(C(K),\ell_{p'})=K(C(K),\ell_{p'})=K(X,Y^*)$$
by \cite[Exercise 6.10]{alka} since $1<p'<2$. Let $x=\frac{1}{2}f$ where $f(t)=1$ for every $t\in [0,1]$, and take $y\in S_Y$. It is clear that $x$ is not an APEP of $B_X$ by Theorem~\ref{theo:C(K)charapep}. However $x\otimes y=f\otimes \left(\frac{1}{2}y\right)$ is an APEP by Theorem~\ref{theo:tensorAPEPcompactneigh} since $f$ is an APEP of $B_X$ (Theorem~\ref{theo:C(K)charapep}) and $\frac{1}{2}y$ is an APEP of $B_Y$ by Example~\ref{exam:lpinfidim}.
\end{example}

It is clear that the above example is based on the absence of uniqueness in the representation of an elementary tensor in a projective tensor product. In order to deal with this difficulty, and taking into account the hypotheses in Theorems~\ref{theo:tensornece} and \ref{theo:tensorAPEPcompactneigh}, we will end the section by studying APEPs under the assumption that $K(X,Y^*)=L(X,Y^*)$, in other words, that every $T\in L(X,Y^*)$ is compact. 
If we additionally require that either $X$ or $Y$ has the AP then we get
$$(X\pten Y)^*=L(X, Y^*)=K(X,Y^*)=X^*\iten Y^*.$$

Note that, for instance, this is the case when $X=\ell_p$ and $Y=\ell_{q'}$ with $1\leq q<p<\infty$ and $1/q+1/q'=1$, thanks to Pitt's theorem (see e.g. Proposition 4.49 in \cite{checos}); a version for Lorentz and Orlicz sequence spaces holds too \cite{AO}. Recall also that for a reflexive space $X$ and a Banach space $Y$, one of them with the compact approximation property, the condition $K(X,Y)=L(X,Y)$ is equivalent to the fact that every operator from $X$ to $Y$ attains its norm. This is shown in \cite{DJMC}, extending previous results of Holub and Mujica in the reflexive case.

In the next result we study APEPs of the unit ball of a projective tensor product $X\pten Y$ under the assumption that $L(X,Y^*)=K(X,Y^*)$. The main technique will be the one used in \cite[Theorem 2.1]{ruste82}. As a byproduct, we obtain a description of all preserved extreme points in this case.

\begin{theorem}\label{th:xtimesyAPEPypreext}
Let $X, Y$ be Banach spaces such that $X^*$ or $Y^*$ has the AP and  $K(X,Y^*)=L(X,Y^*)$. Let $z\in B_{X\pten Y}$ with $z\neq 0$. Then:
\begin{itemize}
    \item[a)] $z\in\ape{B_{X\pten Y}}$ if, and only if, $z=x\otimes y$ for some $x\in\ape{B_X}$ and $y\in\ape{B_Y}$. 
    \item[b)] $z\in\preext{B_{X\pten Y}}$ if, and only if, $z=x\otimes y$ for some $x\in\preext{B_X}$ and $y\in\preext{B_Y}$. 
\end{itemize}
\end{theorem}

This lemma is likely known to experts but we include its proof for the sake of completeness, as it extends the weak-topology version \cite[Theorem 2.3]{ggmr23} to the weak* setting.

\begin{lemma}\label{lemma:weak*closuretensor}
    Let $X,Y$ be Banach spaces. If $C\subseteq X^*$ and $D\subseteq Y^*$ are bounded subsets, then
    $$\overline{C}^{w^*}\otimes \overline{D}^{w^*} = \overline{C\otimes D}^{w^*},$$
    considering $X^*\otimes Y^* \subseteq (X\iten Y)^*$.
\end{lemma}

\begin{proof} If $C=\{0\}$ or $D=\{0\}$ then the result is trivial, so assume $C,D\neq \{0\}$.
    First of all, since $C$ and $D$ are bounded we may assume that $$\sup_{x^*\in C} \norm{x^*}=1=\sup_{y^*\in D}\norm{y^*}.$$ We show first that $\overline{C}^{w^*}\otimes \overline{D}^{w^*} \subseteq \overline{C\otimes D}^{w^*}$. For each $x^*\in X^*$, define the operator $T_{x^*} \colon X\iten Y \rightarrow Y$ given by $T_{x^*}(x\otimes y) = x^*(x)y$ for all $x\in X$ and $y\in Y$, which extends by linearity and continuity to the whole of $X\iten Y$. Hence, the adjoint $T_{x^*}^* \colon Y^*\rightarrow (X\iten Y)^*$ is given by $$T_{x^*}^*(y^*) = x^*\otimes y^*\in (X\iten Y)^*, \quad \forall y^*\in Y^*,$$
    and it is weak*-to-weak* continuous. Therefore, 
    $$\{x^*\} \otimes \overline{D}^{w^*} = \ T^*_{x^*}\left(\overline{D}^{w^*}\right) \subseteq \overline{T_{x^*}^*(D)}^{w^*}= \overline{\{x^*\}\otimes D}^{w^*} \subseteq\overline{C\otimes D}^{w^*},$$
    and this holds for every $x^*\in C$. So, $C\otimes \overline{D}^{w^*} \subseteq \overline{C\otimes D}^{w^*}$ and also $ \overline{C}^{w^*}\otimes D \subseteq \overline{C\otimes D}^{w^*}$, by a symmetric argument. Hence, 
    $$\overline{C}^{w^*}\otimes \overline{D}^{w^*}\subseteq \overline{\overline{C}^{w^*}\otimes D }^{w^*}\subseteq \overline{\overline{C\otimes D}^{w^*}}^{w^*}=\overline{C\otimes D}^{w^*}.$$
    
    It remains to prove that $\overline{C}^{w^*}\otimes \overline{D}^{w^*} \supseteq \overline{C\otimes D}^{w^*}$.
    Take $z\in \overline{C\otimes D}^{w^*}$ and pick a net $(x^*_\alpha \otimes y^*_\alpha)_\alpha \subseteq C\otimes D$ converging weak* to $z$. Since $\overline{C}^{w^*}$ and $\overline{D}^{w^*}$ are weak*-compact, we may assume (by taking subnets if necessary) that $(x^*_\alpha)_\alpha$ and $(y^*_\alpha)_\alpha$ converge weak* to some $x^*\in \overline{C}^{w^*}$ and $y^*\in \overline{D}^{w^*}$ respectively. Let us show that $(x^*_\alpha \otimes y^*_\alpha)_\alpha$ converges weak* to $x^*\otimes y^*$ in $(X\iten Y)^*$. It is clear that $((x^*_\alpha \otimes y_\alpha^*)(x\otimes y) )_\alpha = (x^*_\alpha(x)y^*_\alpha (y))_\alpha $ converges to $x^*(x)y^*(y)=(x^*\otimes y^*)(x\otimes y)$ for all $x\in X$ and $y\in Y$. By linearity, we also have $(x^*_\alpha\otimes y^*_\alpha)(v)\to (x^*\otimes y^*)(v)$ for $v\in X\otimes Y$. Finally, take $u\in X\iten Y$ and $\varepsilon>0$. On the one hand, pick $v\in X\otimes Y$ such that $\norm{u-v}\leq \frac{\varepsilon}{4}$. Hence, for every $\alpha$, we have
    \begin{align*}
          \left| (x^*_\alpha \otimes y_\alpha^*-x^*\otimes y^*)\left(u-v\right)\right|&\leq  \left(\norm{x^*_\alpha}\norm{y^*_\alpha} + \norm{x^*}\norm{y^*}\right)\frac{\varepsilon}{4}\leq \frac{\varepsilon}{2}.   
    \end{align*}
    On the other hand, pick $\beta$ such that $$\left| (x_\alpha^*\otimes y_\alpha^*)\left(v\right)- (x^*\otimes y^*) \left(v\right) \right|<\frac{\varepsilon}{2}, \quad \forall \alpha\geqn\beta.$$ Then,
    $$\left| (x^*_\alpha \otimes y_\alpha^*-x^*\otimes y^*)(u)\right|<\frac{\varepsilon}{2}+\frac{\varepsilon}{2} =\varepsilon, \quad \forall \alpha\geqn \beta,$$
    so $(x^*_\alpha \otimes y^*_\alpha)_\alpha$ converges weak* to $x^*\otimes y^*$. By uniqueness of the limit, $z=x^*\otimes y^* \in \overline{C}^{w^*}\otimes \overline{D}^{w^*}$.
\end{proof}

Now we can provide the pending proof.

\begin{proof}[Proof of Theorem~\ref{th:xtimesyAPEPypreext}] a) Thanks to Theorem \ref{theo:tensorAPEPcompactneigh} we know that if $x$ is an APEP of $B_X$ and $y$ is an APEP of $B_Y$, then $x\otimes y$ is an APEP of $B_{X\pten Y}$. Conversely, assume that $z$ is an APEP of $B_{X\pten Y}$. By Theorem \ref{theo:tensornece}, $z\in B_X\otimes B_Y$. It remains to be shown that $z=x\otimes y$ for some $x\in\ape{B_X}$ and $y\in\ape{B_Y}$. By virtue of Theorem \ref{theo:caraAPEP}, that is equivalent to $x\in \overline{\ext{ B_{X^{**}}}}^{w^*}$ and $y\in \overline{\ext{B_{Y^{**}}}}^{w^*}$. Furthermore, since $z\in B_{X\pten Y}$ is an APEP we have $z \in \overline{\ext{B_{(X\pten Y)^{**}}}}^{w^*}\cap B_{X\pten Y}$. Therefore, let us show that \begin{equation}\label{Eq:xtimesyAPEP}
    \overline{\ext{B_{X^{**}}}}^{w^*}\otimes \overline{\ext{B_{Y^{**}}}}^{w^*}=\overline{\ext{B_{(X\pten Y)^{**}}}}^{w^*}
\end{equation} 
Indeed, 
$$ \overline{\ext{B_{X^{**}}}}^{w^*}\otimes \overline{\ext{B_{Y^{**}}}}^{w^*}= \overline{\ext{B_{X^{**}}}\otimes \ext{B_{Y^{**}}}}^{w^*},$$
using Lemma \ref{lemma:weak*closuretensor}. Finally, $\ext{B_{X^{**}}}\otimes \ext{B_{Y^{**}}}= \ext{B_{(X^*\iten Y^*)^*}}$ by \cite[Theorem 1.1]{ruste82} and $X^*\iten Y^*=K(X,Y^*)=L(X,Y^*)=(X\pten Y)^*$, thanks to \cite[Corollary 4.13]{Ryan}. Hence, 
$$\overline{\ext{B_{(X^*\iten Y^*)^*}}}^{w^*}= \overline{\ext{B_{(X\pten Y)^{**}}}}^{w^*}.$$
This proves (\ref{Eq:xtimesyAPEP}).

Finally, since $z$ is an APEP of $B_{X\pten Y}$, by (\ref{Eq:xtimesyAPEP}) we have $z=x^{**}\otimes y^{**}$ for  some $x^{**}\in \overline{\ext{B_{X^{**}}}}^{w^*}\subseteq B_{X^{**}}$ and $y^{**}\in \overline{\ext{B_{Y^{**}}}}^{w^*}\subseteq B_{Y^{**}}$. Moreover, thanks to Theorem \ref{theo:tensornece}, we also have $z\in B_X\otimes B_Y$. It follows easily that $x^{**}\in B_X$ and $y^{**}\in B_Y$, because $z\neq 0$. Thus, $x^{**}$ is an APEP of $B_X$ and $y^{**}$ is an APEP of $B_{Y}$, which 
concludes the proof of a). 

b) The proof follows immediately by the description of the extreme points of $B_{(X\pten Y)^{**}}$ given before.
\end{proof}

Now the following remark is pertinent.

\begin{remark}\label{remark:solucionggmr23}\hfill
\begin{enumerate}
\item In \cite[Question 3.9]{ggmr23} it is asked whether $x\otimes y$ is a preserved extreme point of $B_{X\pten Y}$ when $x\in B_X$ and $y\in B_Y$ are preserved extreme points of $B_X$ and $B_Y$ respectively. Theorem~\ref{th:xtimesyAPEPypreext} gives an affirmative answer under the assumption that $L(X,Y^*)=K(X,Y^*)$ and that either $X^*$ or $Y^*$ has the AP.

\item In connection with the above question, in \cite[Example 3.8]{ggmr23}, an equivalent norm $\vert\cdot\vert$ on $\ell_2$  and a point $x_0\in B_{(X,\vert\cdot\vert)}$ are given such that $x_0$ is a weakly strongly exposed point (in particular, it is a preserved extreme point) by a certain functional $f\in S_{X^*}$ such that $x_0\otimes x_0$ is not a weakly strongly exposed point.

It is a natural question to ask whether $x_0\otimes x_0$ is a preserved extreme point, and indeed it seems to be the first example to check in order to look for a negative answer to \cite[Question 3.9]{ggmr23} above.

Let us point out, however, that $x_0\otimes x_0$ is an extreme point. Indeed, it is not difficult to prove that $x_0\otimes x_0$ is an exposed point (by $f\otimes f$). Moreover, it can be proved that $x_0\otimes x_0\in \overline{\dent{B_{(X,\vert\cdot\vert)\pten (X,\vert\cdot\vert)}}}$, so in particular it is an APEP. If $x_0\otimes x_0$ is not a preserved extreme point, then this furnishes another example of an extreme, almost preserved-extreme point that is not preserved extreme.
\end{enumerate}
\end{remark}

We finish by extending Theorem~\ref{th:xtimesyAPEPypreext} from the case of the unit ball of a projective tensor product to the case of $\cconv(C\otimes D)$ when $C$ and $D$ have non-empty interior.

\begin{theorem}\label{theo:apepyprextinternova}
Let $X$ and $Y$ be two Banach spaces such that $L(X,Y^*)=K(X,Y^*)$ and such that either $X^*$ or $Y^*$ has the AP. Let $C\subseteq X$ and $D\subseteq Y$ be two bounded, closed, convex and symmetric subsets with non-empty interior. 
Given $z\in \cconv(C\otimes D)$ with $z\neq 0$ we have:
\begin{itemize}
    \item[a)] $z\in\ape{\cconv(C\otimes D)}$ if, and only if, $z=x\otimes y$ for some $x\in\ape{C}$ and $y\in\ape{D}$.
     \item[b)] $z\in\preext{\cconv(C\otimes D)}$ if, and only if, $z=x\otimes y$ for some $x\in\preext{C}$ and $y\in\preext{D}$.
\end{itemize}
\end{theorem}

\begin{proof}
Let us begin by observing that $0$ is an interior point of $C$ (the same holds for $D$). Indeed, take any interior point $u\in C$ and  $\delta>0$ such that $u+\delta B_X^0\subseteq C$, where $B_X^0$ stands for the open unit ball. By symmetry, we also have $-u+\delta B_X^0\subseteq C$. It follows that $\delta B_X^0 \subseteq C$ since $C$ is convex. 

Now, by the properties of $C$, we have that $C$ is the unit ball of some equivalent norm on $X$. We denote by $\tilde X$ such an equivalent renorming for which $B_{\tilde X}=C$. Similarly, we define $\tilde Y$ to be an equivalent renorming of $Y$ such that $B_{\tilde Y}=D$. Observe that by the assumptions either $\tilde X$ or $\tilde Y$ has the AP, and that $L(\tilde X,\tilde Y^*)=K(\tilde X,\tilde Y^*)$. Finally, note that
$$B_{\tilde X\pten \tilde Y}=\overline{\co}(B_{\tilde X}\otimes B_{\tilde Y})=\overline{\co}(C\otimes D).$$
Now the conclusion follows by Theorem~\ref{th:xtimesyAPEPypreext} since $X\pten Y$ and $\tilde X\pten \tilde Y$ are isomorphic Banach spaces, so the weak$^*$ topologies of their corresponding biduals are the same.
\end{proof}

\begin{remark}\label{remark:punturuesstegall}
The above result should be compared with \cite[Theorem 2.2]{ruste82}, which generalises \cite[Theorem 1.1]{ruste82} from the unit ball of the dual of a space of operators to certain weak$^*$ compact neighbourhoods of $0$ in such space.
\end{remark}

\medskip
\section{Concluding remarks and open questions}\label{section:remarks}

In this section we collect some comments and open questions which are derived from our work.

From our study of APEPs in the unit ball of $\ell_1$-sums of spaces, the following remains open.

\begin{question}
Let $\{X_i:i\in I\}$ an arbitrary infinite  family of Banach spaces and let $X$ be its $\ell_1$-sum.
If $0\in B_X$ is an APEP in $B_X$, must there exist $i$ such that $0\in X_i$ is an APEP in $B_{X_i}$?
\end{question}

For the case of $\ell_p$-sums, in Proposition~\ref{prop:APEPlpsumas} we have characterised when the norm-one elements are APEPs of the unit ball. However, we were unable to describe those APEPs whose norm is strictly smaller than 1.

\begin{question}
Let $X$ be the $\ell_p$-sum of a family $\set{X_i:i\in I}$ of Banach spaces, where $1<p<\infty$, and let $(x_i)\in B_X$ with $\Vert (x_i)\Vert<1$. When is $(x_i)$ an APEP of $B_X$?
\end{question}

Even though we do not have a complete characterisation, let us now present a sufficient condition concerning the above question.

\begin{proposition}
Let $X$ be the $\ell_p$-sum of a family $\set{X_i:i\in I}$ of Banach spaces, where $1<p<\infty$, and let $(x_i)\in B_X$ with $\Vert (x_i)\Vert<1$. Set $J:=\{i\in I: \Vert x_i\Vert\neq 0\}$. Assume that there exists $i_0\in I\setminus J$ such that $0$ is an APEP in $B_{X_{i_0}}$ and that $\frac{x_j}{\Vert x_j\Vert}$ is an APEP in $B_{X_j}$  for every $j\in J$. Then $(x_i)$ is an APEP in $B_X$.    
\end{proposition}

\begin{proof} The proof will be quite similar to that of Proposition~\ref{prop:APEPlpsumas}. Let $U$ be a weak$^*$ neighbourhood of $(x_i)$ in $X^{**}$. We will show that $U$ intersects $\ext{B_{X^{**}}}$ and this will be enough by Theorem~\ref{theo:caraAPEP}. 
Since the weak$^*$ topology of $X^{**}$ is the product topology of the weak$^*$ topologies in $X_i^{**}$, we may assume that $U=\prod_{i\in I} U_i$ where $U_i$ is a weak$^*$ neighbourhood of $x_i$ (or $X_i^{**}$) for each $i\in I$. Given $j\in J$, since $x_j\neq 0$ then, by assumption and Theorem~\ref{theo:caraAPEP}, there exists $e_j\in\ext{B_{X_j^{**}}}$ such that $\norm{x_j}e_j\in U_j$. Moreover, since $0$ is an APEP of $B_{X_{i_0}}$ we may select $e_{i_0}\in \ext{B_{X_{i_0}^{**}}}$ such that $\left(1-\Vert (x_i)\Vert^p\right)^\frac{1}{p}e_{i_0}\in U_{i_0}$. Let $y=(y_i)\in X^{**}$ be defined by
$$
y_i = \begin{cases}
\norm{x_i}e_i &\text{, if } i\in J\\
\left(1-\Vert (x_i)\Vert^p\right)^\frac{1}{p}e_{i_0} & \text{, if }i=i_0\\
0 &\text{, otherwise.}
\end{cases}
$$
Note that $y\in U$ and
\[
\begin{split}\norm{y} = \pare{\sum_{i\in I}\norm{y_i}^p}^{1/p}&  = \pare{\sum_{j\in J}\norm{x_j}^p+\pare{\pare{1-\Vert (x_i)\Vert^p}^\frac{1}{p}}^p}^{1/p}\\
& =\pare{\sum_{j\in J}\Vert x_j\Vert^p+1-\sum_{j\in J}\Vert x_j\Vert^p}^\frac{1}{p} = 1,
\end{split}
\]
 so $y\in \ext{B_{X^{**}}}$ by \eqref{eq:extreme_ell_p}. This ends the proof.
\end{proof}

Another question coming from Subsection~\ref{subsection:absolutesums} is the following.

\begin{question}
Let $\{X_i:i\in I\}$ be an arbitrary infinite family of Banach spaces and let $X$ be its $\ell_\infty$-sum. Can necessary or sufficient conditions for a point $(x_i)\in B_X$ to be an APEP be given?
\end{question}

Regarding Lipschitz-free spaces, in Section~\ref{section:Lipschitzfree} we proved that every APEP of $B_{\F{M}}$ must be either a molecule or $0$, and we characterised those molecules that are APEPs. We showed that $0$ can be an APEP, but only if $M$ does not bi-Lipschitz embed in $\mathbb R^n$, and that this necessary condition is not sufficient (Examples \ref{example:0apepfree} and \ref{example:no0apepfree}). The following remains open.

\begin{question}
For which metric spaces $M$ is $0$ an APEP of $B_{\mathcal F(M)}$?
\end{question}

Finally, let us collect some open questions from our results in Section~\ref{section:tensor}. First of all, concerning Theorem~\ref{theo:tensornece} we wonder the following.

\begin{question}
In the statement of Theorem~\ref{theo:tensornece}, can we infer that if $z\neq 0$ then $x$ and $y$ can be chosen to be APEPs in $C$ and $D$ respectively?    
\end{question}

Example~\ref{example:tensorAPEPfactorno} shows that not every such representation $z=x\otimes y$ satisfies $x\in \ape{C}$ and $y\in \ape{D}$.

Our last question has to do with the possibility of removing the assumption of the existence of a compact neighbourhood system for the weak topology in Theorem~\ref{theo:tensorAPEPcompactneigh}.

\begin{question}
Let $X$ and $Y$ be Banach spaces, and let $C\subseteq X$, $D\subseteq Y$ be symmetric bounded closed convex subsets. If $x_0\in C$ and $y_0\in D$ are APEPs, is $x_0\otimes y_0$ an APEP of $\cconv(C\otimes D)$? Can we obtain this at least when $C=B_X$ and $D=B_Y$?
\end{question}

We conclude this section by presenting an alternative proof of Theorem~\ref{theo:cararnpapepsubconjuntos}. This proof, inspired by a construction due to W. Schachermayer, A. Sersouri, and E. Werner \cite{scsewe89}, offers a more geometric perspective, in contrast to the measure-theoretic approach developed in Section~\ref{section:charnp}.

We need the following elementary result, whose proof is included for completeness.

\begin{lemma}\label{lemma:disw*closure}
Let $C\subseteq X$ be a convex set and let $x\in X$ such that $x\notin\overline{C}$. Then
$$\dist(x,C)=\dist(x,\overline{C}^{w^*}),$$
where the weak* closure above is taken in $X^{**}$.    
\end{lemma}

\begin{proof} Take any $0 < r < \dist(x,C)$. Then $\overline{C}$ and $B(x,r)$ can be separated by a hyperplane determined by some $x^* \in X^*$. Then the same hyperplane separates $\overline{B(x,r)}^{w^*}$ and $\overline{C}^{w^*}$ in $X^{**}$, thus $\dist(x,\overline{C}^{w^*}) \geq r$.
\end{proof}

Following \cite{scsewe89}, for a measurable subset $A\subseteq[0,1]$ with positive measure, we denote
\[ \mathcal F_A = \{f\in L_1[0,1] : f= f\cdot \chi_A, f \geq 0, \norm{f}_1 =1\}.\]

\begin{proof}[Second proof of Theorem \ref{theo:cararnpapepsubconjuntos}]
First, note that we may assume that $X$ is separable. Indeed, if $X$ fails the RNP then there is a separable $Y\subseteq X$ failing the RNP. Now, for a bounded closed convex set $C\subseteq Y$ and a point $x_0\in C$, we have that $x_0\in \ape{C}$ as a subset of $Y$ if and only if $x_0\in \ape{C}$ as a subset of $X$, since the weak topology in $Y$ coincides with the weak topology inherited from $X$.

Now, let $X$ be a separable Banach space failing the RNP. We will show that there is a bounded closed convex set $C\subseteq X$ with $\dist\left(X, \overline{\ext{\overline{C}^{w^*}}}^{w^*}\right)>0$ (and so, by Theorem \ref{theo:caraAPEP}, $\ape{C}=\emptyset$).

Since $X$ fails the RNP, there is a non-representable operator $T\colon L_1[0,1]\to X$. Let $(y_n)_{n=1}^\infty$ be a dense sequence in $X$. The proof of Theorem 1.1 in \cite{scsewe89} shows that there exists $\gamma >0$ and measurable subsets $D_1^n,\ldots, D_{N(n)}^n\subseteq [0,1]$ such that, if we denote $E:=\bigcap_{n=1}^\infty \bigcup_{i=1}^{N(n)} D_i^n$ and $E_i^n = E\cap D_i^n$, we have:
\begin{itemize}
\item[a)] $\dist(y_n, T(\mathcal F_{D_i^n}))>\gamma$ for all $i\in\{1,\ldots, N(n)\}$ and $n\in \mathbb N$.
\item[b)] For each $n\in\mathbb N$, every extreme point $y^{**}$ of  $\overline{T(\mathcal F_E)}^{w^*}$ belongs to $\overline{T(\mathcal F_{E_i^n})}^{w^*}$ for some $i\in \{1,\ldots, N(n)\}$.
\end{itemize}
We claim that $C=\overline{T(\mathcal F_E)}$ satisfies the desired property. First, note that $C$ is bounded, closed and convex (since $\mathcal F_E$ is a convex set). Now, let $y^{**}\in \overline{\ext{\overline{C}^{w^*}}}^{w^*}$ and take a net $(y_s^{**})_s\subseteq \ext{\overline{C}^{w^*}}$ with $y_s^{**}\stackrel{w^*}{\to} y^{**}$. Fix $n\in \mathbb N$. By property b) above, there is $i_n\in\{1,\ldots, N(n)\}$ such that the net $(y_s^{**})_s$ is frequently in $\overline{T(\mathcal F_{E_{i_n}^n})}^{w^*}$. Thus, $y^{**}\in \overline{T(\mathcal F_{E_{i_n}^n})}^{w^*}$. By Lemma \ref{lemma:disw*closure}, we get for all $n$
\begin{align*} \dist(y_n, y^{**}) &\geq \dist\left(y_n, \overline{T(\mathcal F_{E_{i_n}^n})}^{w^*}\right) = \dist(y_n,  T(\mathcal F_{E_{i_n}^n}))\\
&\geq \dist(y_n,  T(\mathcal F_{D_{i_n}^n})) > \gamma\end{align*}
and so $\dist(X, y^{**})\geq \gamma$. This proves the desired claim.  
\end{proof}

\medskip
\section*{Acknowledgements} 

The authors are grateful to Miguel Mart\'in for fruitful conversations on the topic of the paper. Moreover, the authors are deeply grateful to the anonymous referee for pointing out a mistake in a previous version of the paper and for providing a big number of suggestions which resulted in a major improvement of the exposition of the results.

R. J. Aliaga was partially supported by Grant PID2021-122126NB-C33 funded by MICIU/AEI/10.13039/501100011033 and by ERDF/EU.

The research of Luis C. García-Lirola was supported by grants PID2021-122126NB-C31 and PID2022-137294NB-I00 funded by MCIN/AEI/\\ 
10.13039/501100011033 and by ``ERDF A way of making Europe''; and by grant E48-23R funded by Diputación General de Aragón (DGA).

The research of Juan Guerrero-Viu was supported by grant PID2022-137294NB-I00 funded by MCIN/AEI/
10.13039/501100011033 and by ``ERDF A way of making Europe''; and by grant E48-23R funded by Diputación General de Aragón (DGA).

The research of Matías Raja was supported by grants PID2021-122126NB-C31 funded by MCIN/AEI/ 
10.13039/501100011033 and by ``ERDF A way of making Europe''; and by Fundaci\'on S\'eneca: ACyT Regi\'on de Murcia grant 21955/PI/22.

The research of A. Rueda Zoca was supported  by MCIU/AEI/\\ FEDER/UE Grant PID2021-122126NB-C31, by MICINN (Spain) Grant \\ CEX2020-001105-M (MCIU, AEI), by Junta de Andaluc\'{\i}a Grant FQM-0185 and by Fundaci\'on S\'eneca: ACyT Regi\'on de Murcia grant 21955/PI/22.

\end{document}